\documentclass{amsart}
\usepackage{amsmath}
\usepackage{amssymb}
\usepackage{amsfonts}
\usepackage{float}
\usepackage{fancyhdr}

\newtheorem{theorem}{Theorem}[section]
\theoremstyle{plain}

\newtheorem{corollary}[theorem]{Corollary}
\newtheorem{lemma}[theorem]{Lemma}
\newtheorem{definition}[theorem]{Definition}
\newtheorem{proposition}[theorem]{Proposition}
\newtheorem{remark}[theorem]{Remark}
\numberwithin{equation}{section}
\input{tcilatex}
\begin{document}
\title{Factorizations of finite groups by conjugate subgroups which are
solvable or nilpotent}
\author{Martino Garonzi}
\address[Martino Garonzi]{Departamento de Matematica\\
Universidade de Bras\'{\i}lia\\
Campus Universit\'{a}rio Darcy Ribeiro\\
Bras\'{\i}lia - DF 70910-900\\
Brasil}
\email{mgaronzi@gmail.com}
\thanks{M.G. acknowledges the support of MIUR-Italy via PRIN Group theory
and applications.}
\author{Dan Levy}
\address[Dan Levy]{The School of Computer Sciences \\
The Academic College of Tel-Aviv-Yaffo \\
2 Rabenu Yeruham St.\\
Tel-Aviv 61083\\
Israel}
\email{danlevy@mta.ac.il}
\author{Attila Mar\'{o}ti}
\address[Attila Mar\'{o}ti]{Fachbereich Mathematik, Technische Universit\"{a}%
t Kaiserslautern\\
Postfach 3049, 67653 Kaiserslautern\\
Germany \\
and Alfr\'{e}d R\'{e}nyi Institute of Mathematics\\
Re\'{a}ltanoda utca 13-15\\
H-1053, Budapest\\
Hungary}
\email{maroti@mathematik.uni-kl.de and maroti.attila@renyi.mta.hu}
\thanks{A.M. acknowledges the support of an Alexander von Humboldt
Fellowship for Experienced Researchers and the support of OTKA K84233.}
\author{Iulian I. Simion}
\address[Iulian I. Simion]{Department of Mathematics\\
University of Padova\\
Via Trieste 63\\
35121 Padova\\
Italy}
\email{iulian.simion@math.unipd.it}
\thanks{I.S. acknowledges the support of the University of Padova (grants
CPDR131579/13 and CPDA125818/12). }
\date{\today}

\begin{abstract}
We consider factorizations of a finite group $G$ into conjugate subgroups, $%
G=A^{x_{1}}\cdots A^{x_{k}}$ for $A\leq G$ and $x_{1},\ldots ,x_{k}\in G$,
where $A$ is nilpotent or solvable. First we exploit the split $BN$-pair
structure of finite simple groups of Lie type to give a unified
self-contained proof that every such group is a product of four or three
unipotent Sylow subgroups. Then we derive an upper bound on the minimal
length of\ a solvable conjugate factorization of a general finite group.
Finally, using conjugate factorizations of a general finite solvable group
by any of its Carter subgroups, we obtain an upper bound on the minimal
length of a nilpotent conjugate factorization of a general finite group.
\end{abstract}

\maketitle

\section{Introduction\label{intro}}

In this paper we continue the study of minimal length factorizations of
(mainly finite) groups into products of conjugate subgroups, that was
initiated in \cite{GL2014conjugates} and \cite{CGLMS2014conjugates}. For a
group $G$ and $A\leq G$, a conjugate product factorization (a $\func{cp}$%
-factorization) of length $k$ of $G$ by $A$, is a factorization $%
G=A_{1}\cdots A_{k}$ where $A_{1},\dots ,A_{k}$ are all conjugate to $A$ and
the product is the setwise product. Denoting the normal closure of $A$ in $G$
by $A^{G}$, an elementary argument shows that $A^{G}$ is equal to a product
of conjugates of $A$ (see \cite{KL2007products}). Thus, a necessary
condition (and for finite groups also a sufficient condition) for the
existence of the factorizations we are interested in, is $G=A^{G}$.

\begin{definition}
Let $G$ be a group and $A\leq G$. Then $\gamma _{\func{cp}}^{A}(G)$ is the
smallest number $k$ such that $G$ equals a product $A_{1}\cdots A_{k}$ of
conjugates of $A$ or $\infty $ if no such $k$ exists. We also set 
\begin{equation*}
\gamma _{\func{cp}}(G):=\min \{\gamma _{\func{cp}}^{A}(G)|A<G\}\text{,}
\end{equation*}%
where $\min \left\{ \infty \right\} :=\infty $.
\end{definition}

It is easy to see that $\gamma _{\func{cp}}^{A}(G)\geq 3$ if $A<G$ (\cite[%
Lemma 6]{GL2014conjugates}). For finite, non-nilpotent, solvable groups $%
\gamma _{\func{cp}}(G)\leq 4\log _{2}|G|$ \footnote{%
In fact, this bound is somewhat improved by Theorem \ref{Th_Carter} below.}
and no universal constant upper bound exists for all $G$ (see \cite[Theorems
4 and 5]{GL2014conjugates}). In contrast, if $G$ is a finite non-solvable
group, then $\gamma _{\func{cp}}(G)=3$ (\cite{CGLMS2014conjugates}).
Moreover, in \cite{CGLMS2014conjugates} we proved that $\gamma _{\func{cp}%
}(G)=3$ holds for any group $G$ with a $BN$-pair and a finite Weyl group,
and this optimal result arises from choosing $A$ to be the Borel subgroup of 
$G$ which is in particular solvable. This motivated us to diversify the
analysis of $\func{cp}$-factorizations in the present paper by imposing
conditions on the subgroup $A$.

Let $G$ be a finite simple group of Lie type whose defining characteristic
is $p$. The first problem we consider is the $\func{cp}$-factorization of $G$
by Sylow $p$-subgroups (also called unipotent Sylows). Liebeck and Pyber
have proved (\cite[Theorem D]{LP2001linear}) that $G$ is a product of no
more than $25$ Sylow $p$-subgroups. Several papers considering the same
question then followed. In \cite{BNP2008mixing} it was claimed that the $25$
can be replaced by $5$, however a complete proof has not been published. A
sketch of a proof of this claim for exceptional Lie type groups appears in a
survey by Pyber and Szab\'{o} (\cite[Theorem 15]{PS2012growth}). Smolensky,
Sury and Vavilov (\cite[Theorem 1]{VSS2011unitriangular}) consider the
problem of unitriangular factorizations of Chevalley groups over commutative
rings of stable rank $1$. When specializing their results to elementary
Chevalley groups over finite fields, they get that any non-twisted finite
simple group of Lie type is a product of four unipotent Sylows. Later on,
these results were extended by Smolensky in \cite{smolensky2013twisted} to
cover some twisted Chevalley groups over finite fields or the field of
complex numbers.

In Section \ref{Section_uniSylow} we give a unified self-contained treatment
of the problem of finding minimal length products of unipotent Sylows for
all finite simple groups of Lie type, exploiting their split $BN$-pair
structure. Recall that a Carter subgroup of a finite group $G$ is a
self-normalizing nilpotent subgroup (see \cite{Carter1961carter}).

\begin{theorem}
\label{Th_unipotent_Sylows}Let $G$ be a simple group of Lie type whose
defining characteristic is $p$. Let $U$ be a Sylow $p$-subgroup of $G$ and
set $U^{-}:=U^{n_{o}}$, where $n_{0}$ is a representative of the longest
element of the associated Weyl group.\ Then $G=\left( UU^{-}\right) ^{2}$,
and $G=UU^{-}U$ if and only if $U$ is a Carter subgroup. In both cases these
factorizations are of minimal length.
\end{theorem}

After the completion of our proof of Theorem \ref{Th_unipotent_Sylows}, and
in parallel to its publication in preprint form (\cite{GLMSsmallU42015}),
Smolensky made available a preprint in which he shows that every Suzuki and
Ree group is a product of four unipotent Sylow subgroups (\cite%
{SmolenskySuzRee2015}). Thus, the results in \cite{tavgen1990bounded},\cite%
{tavgen1992bounded},\cite{VSS2011unitriangular},\cite{smolensky2013twisted}
and \cite{SmolenskySuzRee2015}, combine to give a different proof of the
four Sylow claim of Theorem \ref{Th_unipotent_Sylows}.

Theorem \ref{Th_unipotent_Sylows} and the $\func{cp}$-factorization of
finite simple groups of Lie type by Borel subgroups mentioned above clearly
motivate a general study of $\func{cp}$-factorizations of a finite group $G$
by subgroups which are nilpotent or solvable. Observe that if $G$ is any
finite group then it is equal to the product of a finite number of
conjugates of a nilpotent (hence also solvable) $A\leq G$. For if $G$ is
nilpotent or $A<G$ is both nilpotent and non-normal maximal, the claim is
clear. Otherwise $G=A_{1}\cdots A_{k}$ is a $\func{cp}$-factorization where $%
A_{1}$ is some maximal non-normal subgroup of $G$, and by induction $A_{1}$
has a nilpotent $\func{cp}$-factorization. Therefore, for a finite group $G$
the following quantities are always natural numbers:%
\begin{equation*}
\gamma _{\func{cp}}^{\func{s}}(G):=\min \{\gamma _{\func{cp}}^{A}(G)|A\leq G%
\text{ is solvable}\}\text{.}
\end{equation*}%
\begin{equation*}
\gamma _{\func{cp}}^{\func{n}}(G):=\min \{\gamma _{\func{cp}}^{A}(G)|A\leq G%
\text{ is nilpotent}\}\text{.}
\end{equation*}%
Note that $\gamma _{\func{cp}}^{\func{s}}(G)=1$ ($\gamma _{\func{cp}}^{\func{%
n}}(G)=1$) if and only if $G$ is solvable (nilpotent). In Section \ref%
{Sect_solvable_factorizations} we prove the following upper bound on $\gamma
_{\func{cp}}^{\func{s}}(G)$ in terms of the quantity $|G|_{\func{nab}}$,
defined as the product of the orders of all non-abelian composition factors
of $G$ (if $G$ is solvable, $|G|_{\func{nab}}:=2$).

\begin{theorem}
\label{Th_loglog}For any finite group $G$, we have $\gamma _{\func{cp}}^{%
\func{s}}(G)\leq $ $1+c_{S}(\log _{2}\log _{2}|G|_{\func{nab}})^{2}$, where $%
0<c_{S}\leq 12/\log _{2}(5)<5.17$ is a universal constant.
\end{theorem}

Using Theorem \ref{Th_loglog} we can reduce the problem of obtaining an
upper bound on $\gamma _{\func{cp}}^{\func{n}}(G)$ for a general finite
group $G$ to a solvable $G$. Moreover, it turns out that there is a
judicious choice of a nilpotent subgroup in the case that $G$ is solvable -
that of a Carter subgroup (see above and Lemma \ref{Lem_Carter_properties}).
The following theorem is proved in Section \ref%
{Sect_nilpotent_factorizations}.

\begin{theorem}
\label{Th_Carter} Any finite solvable group $G$ is a product of at most $%
1+c_{A}\log _{2}\left\vert G:C\right\vert $ conjugates of a Carter subgroup $%
C$, where $0<c_{A}\leq 3/\log _{2}(5)<1.3$ is a universal constant.
\end{theorem}

\begin{corollary}
\label{Coro_ConjugateNil}For any finite group $G$ there exists a nilpotent $%
H\leq G$ such that 
\begin{equation*}
\gamma _{\func{cp}}^{\func{n}}(G)\leq \gamma _{\func{cp}}^{H}(G)\leq
1+c_{S}c_{A}\left( \log _{2}|G:H|\right) \left( \log _{2}\log _{2}|G|\right)
^{2}\text{.}
\end{equation*}
\end{corollary}

The proof of Theorem \ref{Th_Carter} requires the following result which is
of independent interest.

\begin{theorem}
\label{Th_AffinePrimitive}Let $G$ be a finite affine primitive permutation
group with a non-trivial point stabilizer $H$. Then $G$ is a product of at
most $1+c_{A}\log _{2}\left\vert G:H\right\vert $ conjugates of $H$, where $%
0<c_{A}\leq 3/\log _{2}(5)<1.3$ is a universal constant.
\end{theorem}

A natural question to ask, in view of the last theorem, is to what extent
can it be generalized to an arbitrary finite primitive permutation group. We
plan to consider this question separately (\cite{GLMSInProgress}).

\section{Factorizations by unipotent Sylow subgroups\label{Section_uniSylow}}

The proof of Theorem \ref{Th_unipotent_Sylows} consists of two main steps:
1. A reduction to the case of a two element Weyl group (i.e., $W=Z_{2}$, and
we shall say that $G$ is a "rank 1 group"), which is carried within the
framework of groups with a split $BN$-pair, and some extra assumptions to be
detailed in the sequel. 2. A derivation of a general necessary and
sufficient criterion for rank $1$ groups satisfying an extended set of split 
$BN$-pair assumptions, that is then verified to hold for the special case of
groups with a $\sigma $-setup, using a result from \cite{CEG2000gauss}.

We would like to point out that although the proof of \cite[Theorem 1]%
{VSS2011unitriangular} also uses a "reduction to rank 1 argument" which is
due to Tavgen$^{\prime }$ (\cite{tavgen1990bounded}), we do not know if
there is a more direct relation between this approach and ours.

We treat simple groups of Lie type in the setting of groups with a $\sigma $%
-setup as in \cite[Definition 2.2.1]{GLS1998number3}. For this fix a prime $%
p $, a simple algebraic group $\overline{K}$ defined over $\overline{\mathbb{%
F}}_{p}$ and a Steinberg endomorphism $\sigma $ of $\overline{K}$, and
consider $K$ - the subgroup of $C_{\overline{K}}(\sigma )$ generated by all $%
p$-elements. All groups $K$ obtained in this way are said to have a $\sigma $%
-setup given by the pair $(\overline{K},\sigma )$. The set of all groups
possessing a $\sigma $-setup for the prime $p$ is denoted by ${\mathcal{L}ie}%
(p)$. We have surjective homomorphisms $K_{u}\longrightarrow
K\longrightarrow K_{a}$ with central kernels (\cite[Theorem 2.2.6]%
{GLS1998number3}), where the groups $K_{u}$, $K_{a}\in {\mathcal{L}ie}(p)$
are called the universal and the adjoint version of $K$, respectively. Due
to the classification of simple groups of Lie type, if we exclude the Tits
group ${}^{2}F_{4}(2)^{\prime }$, then all such groups lie in ${\mathcal{L}ie%
}(p)$ for some $p$ appearing as the adjoint version $K_{a}$ of some $K$. Set 
${\mathcal{L}ie:}=\cup {\mathcal{L}ie}(p)$ the union over all primes $p$.
For $G\in {\mathcal{L}ie}$ {we have (see \cite[Chapter 2]{carter1985finite}):%
}

\begin{enumerate}
\item[(i)] $G$ is a group with a split $BN$-pair $(B,N)$ and a finite Weyl
group $W$, where $B=H\ltimes U$,

\item[(ii)] $U$ is a Sylow $p$-subgroup of $G$,

\item[(iii)] $G$ is generated by its $p$-elements.
\end{enumerate}

\subsection{Reduction to $\left\vert W\right\vert =2$ case for groups with a
split $BN$-pair\label{SubSect_splitBNpairs}}

For our purposes we will call a triple $(H,U,N)$ a split $BN$-pair for a
group $G$ if $(H\ltimes U,N)$ satisfies the axioms of split $BN$-pairs in 
\cite[\S 2.5]{carter1985finite} with respect to $H$ and $U$. We assume that
the Weyl group $W=N/H$ of the $BN$-pair is finite (this certainly holds for
finite groups), and so the longest element $w_{0}=n_{0}H$ of $W$ exists and
defines subgroups $U^{-}:=U^{n_{o}}$ and $B^{-}:=B^{n_{o}}$. For $w\in W$ we
sometimes use $\dot{w}$ to denote an arbitrary choice of an element of $N$
such that $w=\dot{w}H$. We use the notation $U^{-}$, $X_{i}$, $U_{i}$, $%
X_{-i}$, $U_{w}$ from \cite[\S 2.5]{carter1985finite} for the $BN$-pair $%
(B,N)$ and we label with the upper-script `$-$' the corresponding subgroups
for $(B^{-},N)$, i.e. when $U$ and $B$ are replaced by $U^{-}$ and $B^{-}$
everywhere: $(U^{-})^{-}$, $X_{i}^{-}$, $U_{i}^{-}$, $X_{-i}^{-}$, $%
U_{w}^{-} $. Note that $(U^{-})^{-}=U$, $X_{i}^{-}=X_{-i}$, $%
X_{-i}^{-}=X_{i} $ and that $X_{i}=U_{s}$ and $X_{-i}=U_{s}^{-}$ for some
simple reflection $s $. In addition, define $L_{s}:=\langle
U_{s},U_{s}^{-},H\rangle $ and $G_{s}:=\langle U_{s},U_{s}^{-}\rangle $ for
any simple reflection $sH\in W$. Furthermore, we assume that the root
subgroups $X_{\alpha }$, $\alpha \in \Phi $ ($\Phi $ is the set of roots
associated with $W$) satisfy the commutator relations (\cite[p.61]%
{carter1985finite}).

\begin{lemma}
\label{Lem_reduction_to_UV} Let $G$ be a group with a split $BN$-pair $%
(H,U,N)$. Suppose that $G$ is a product of $k$ conjugates of $U$ ($k\geq 3$
an integer). Then 
\begin{equation*}
G=UU^{-}UU^{-}\cdots UU^{-}U^{\varepsilon }\text{,}
\end{equation*}%
where $\varepsilon \in \left\{ +,-\right\} $, $U^{+}=U$, the total number of
non-trivial conjugates of $U$ which appear on the r.h.s. is $k$, and $%
\varepsilon =+$ if and only if $k$ is odd.
\end{lemma}

\begin{proof}
We have $\gamma _{\func{cp}}^{U}(G)=k$ $\Leftrightarrow $ $%
G=U^{x_{1}}U^{x_{2}}\cdots U^{x_{k}}$ for some elements $x_{1},\ldots
,x_{k}\in G$. This is equivalent to $G=Ug_{1}U\cdots Ug_{k-1}U$ for some $%
g_{1},\ldots ,g_{k-1}\in G$ (see \cite[\S 2 Lemma 1]{CGLMS2014conjugates}).
By the Bruhat expression of elements (w.r.t. $H$ and $U$) we may assume that 
$g_{i}\in N$ for all $i$. Indeed, by \cite[Theorem 2.5.14]{carter1985finite}%
, for each $1\leq i\leq k-1$, we have $g_{i}=u_{i}h_{i}\dot{w}%
_{i}u_{i}^{\prime }$ where $u_{i}\in U$, $h_{i}\in H$, $w_{i}\in W$ and $%
u_{i}^{\prime }\in U_{w}\leq U$. Since the $h_{i}$ lie in $N_{G}(U)$ we have 
$Ug_{1}U\dots Ug_{k-1}U=Uh_{1}\dot{w}_{1}U\dots Uh_{k-1}\dot{w}_{k-1}U=U\dot{%
w}_{1}U\dots U\dot{w}_{k-1}U(h_{1}^{\dot{w}_{1}\cdots \dot{w}_{k-1}}\cdots
h_{k-1})$ so $\gamma _{\func{cp}}^{U}(G)=k$ is equivalent to $G=Ug_{1}U\dots
Ug_{k-1}U$ for some $g_{1},\ldots ,g_{k-1}\in N$. This is equivalent to $%
G=UU^{g_{1}^{-1}}U^{g_{2}^{-1}}\cdots U^{g_{k-1}^{-1}}$ for some $%
g_{1},...,g_{k-1}\in N$ (see proof of \cite[\S 2 Lemma 1]%
{CGLMS2014conjugates}).

Let $n\in N$ be arbitrary, and let $w=nH$. By \cite[Proposition 2.5.12]%
{carter1985finite} we have 
\begin{equation*}
U=U_{w_{0}w}U_{w}=\left( U\cap U^{n_{0}\left( n_{0}n\right) }\right) \left(
U\cap U^{n_{0}n}\right) \text{,}
\end{equation*}%
which gives 
\begin{equation*}
U^{n^{-1}}=\left( U\cap U^{n_{0}\left( n_{0}n\right) }\right)
^{n^{-1}}\left( U\cap U^{n_{0}n}\right) ^{n^{-1}}=\left( U^{n^{-1}}\cap
U\right) \left( U^{n^{-1}}\cap U^{n_{0}}\right) \leq UU^{-}\text{.}
\end{equation*}%
However, since $U,U_{w_{0}w},U_{w}$ are all subgroups, $U=U_{w_{0}w}U_{w}$
implies $U=U_{w}U_{w_{0}w}$, and hence we also get $U^{n^{-1}}\leq U^{-}U$.
Therefore: 
\begin{equation*}
G=UU^{g_{1}^{-1}}U^{g_{2}^{-1}}\cdots U^{g_{k-1}^{-1}}\subseteq U\left(
UU^{-}\right) \left( U^{-}U\right) \left( UU^{-}\right) \left( U^{-}U\right)
\cdots =\newline
UU^{-}UU^{-}U\cdots \text{,}
\end{equation*}%
where we have used $U^{2}=U$ and $\left( U^{-}\right) ^{2}=U^{-}$, and the
claim follows.
\end{proof}

In the following lemma we collect known results about minimal (non-abelian)
Levi subgroups which will be used in the sequel. First note that for a fixed
split $BN$-pair $(H,U,N)$ we have \emph{the (split) $BN$-pair opposite to $%
(B,N)$} given by $(H,U^{-},N)$. Clearly, for any $g\in G$, $(B^{g},N^{g})$
is a split $BN$-pair, and if $g\in N$ then $B^{g}\cap N=H$ so $%
B^{g}=H\ltimes U^{g}$. In particular this applies to $g=n_{0}$.

\begin{lemma}
\label{min_levi} Let $G$ have a split $BN$-pair $(H,U,N)$. Let $w_{0}$ be
the longest element of the Weyl group $N/H$ and $s=n_{s}H$ a simple
reflection with respect to $(H,U,N)$. Then

\begin{enumerate}
\item[(a)] $U=U_{s}U_{w_{0}s}=U_{w_{0}s}U_{s}$ and $%
U^{-}=U_{s}^{-}U_{w_{0}s}^{-}=U_{w_{0}s}^{-}U_{s}^{-}$,

\item[(b)] $L_{s}\subseteq N_{G}(U_{w_{0}s})\cap N_{G}(U_{w_{0}s}^{-})$.
\end{enumerate}
\end{lemma}

\begin{proof}
Any $s=n_{s}H$ in $W$ is simple with respect to $(H,U,N)$ if and only if it
is simple with respect to $(H,U^{-},N)$: This follows from \cite[%
Propositions 2.2.6 and 2.2.7]{carter1985finite} and the fact that the
positive roots with respect to $(H,U^{-},N)$ are the negative roots with
respect to $(H,U,N)$. So $I:=\left\{ s_{1},...,s_{l}\right\} $, the set of
simple reflections for $(H,U,N)$, is a set of simple reflections for both of
these $BN$ pairs and $s=s_{i}$ for some $i\in \{1,\dots ,l\}$. Since $L_{s}$
is the subgroup $X_{i}H\cup X_{i}Hn_{i}X_{i}=\langle X_{i},X_{-i},H\rangle $
(see \cite[Corollary 2.6.2]{carter1985finite}), $X_{i}^{-}=X_{-i}$ and $%
X_{-i}^{-}=X_{i}$, it follows that $L_{s}=\langle
X_{i}^{-},X_{-i}^{-},H\rangle $ is a (minimal) standard Levi subgroup with
respect to both $(B,N)$ and $(B^{-},N)$.

Now (a) follows from \cite[Proposition 2.5.11]{carter1985finite} and (b) is
a particular case of \cite[Proposition 2.6.4]{carter1985finite}.
\end{proof}

The following lemma is \cite[Lemma 4]{VSS2011unitriangular}.

\begin{lemma}
\label{gen_set} Let $G$ be a group and let $X\subseteq G$ satisfy $X=X^{-1}$
and $G=\langle X\rangle $. If $\emptyset \neq Y\subseteq G$ is such that $%
XY\subseteq Y$ then $Y=G$.
\end{lemma}

\begin{proposition}
\label{reduction_rank1} Let $G$ be a group with a split $BN$-pair such that
the conjugates of $U$ in $G$ generate $G$. Let $k\geq 2$ be an integer and
assume further that $G_{s}=\left(U_{s}U_{s}^{-}\right) ^{k}$ for every
simple reflection $s$ then $G=\left(UU^{-}\right)^{k}$.
\end{proposition}

\begin{proof}
Set $X:=\left\{ u^{g}|u\in U,g\in G\right\} $. Then $X=X^{-1}$ since $U$ is
a subgroup of $G$, and $G=\left\langle X\right\rangle $ since $G$ is the
normal closure of $U$. Set $Y:=\left( UU^{-}\right) ^{k}$. By Lemma \ref%
{gen_set} our claim will follow if we show that $XY\subseteq Y$. Thus it
suffices to show that $u^{g}Y\subseteq Y$ for any $u\in U$ and $g\in G$. By 
\cite[Theorem 2.5.14]{carter1985finite}, for any $g\in G$ there exist $%
u^{\prime }\in U$, $h\in H$, $w\in W$ and $u^{\prime \prime }\in U_{w}\leq U$
such that $g=u^{\prime }hn_{w}u^{\prime \prime }$. Hence $u^{g}=u^{u^{\prime
}hn_{w}u^{\prime \prime }}=\left( u^{u^{\prime }h}\right) ^{n_{w}u^{\prime
\prime }}$. \ But $u^{u^{\prime }h}\in U$, so it is sufficient to prove that 
$u^{nv}Y\subseteq Y$ for all $u,v\in U$ and $n\in N$. Now we claim that the
last statement follows if we prove that $N$ normalizes $Y$. For suppose that 
$N$ normalizes $Y=\left( UU^{-}\right) ^{k}$. We have: 
\begin{eqnarray*}
u^{nv}\left( UU^{-}\right) ^{k} &=&v^{-1}n^{-1}unv\left( UU^{-}\right)
^{k}=v^{-1}n^{-1}un\left( UU^{-}\right) ^{k}= \\
&=&v^{-1}n^{-1}u\left( UU^{-}\right) ^{k}n=v^{-1}n^{-1}\left( UU^{-}\right)
^{k}n= \\
&=&v^{-1}\left( UU^{-}\right) ^{k}n^{-1}n=\left( UU^{-}\right) ^{k}\text{.}
\end{eqnarray*}

Thus we prove that $N$ normalizes $\left( UU^{-}\right) ^{k}$. Since $H$
clearly normalizes $\left( UU^{-}\right) ^{k}$, and $N$ is generated by a
set $I$ of representatives for simple reflections together with $H$, it is
sufficient to prove that $\left( UU^{-}\right) ^{k}$ is normalized by all $n$
in $I$. Fix a simple reflection $s=nH$. By Lemma \ref{min_levi}.(a), $%
UU^{-}=U_{s}U_{w_{0}s}U_{w_{0}s}^{-}U_{s}^{-}$. By Lemma \ref{min_levi}.(b),
each of $U_{s}$ and $U_{s}^{-}$ commutes with both $U_{w_{0}s}$ and $%
U_{w_{0}s}^{-}$. This, and the assumption $G_{s}=\left(
U_{s}U_{s}^{-}\right) ^{k}$, give: 
\begin{equation*}
\left( UU^{-}\right) ^{k}=\left( U_{s}U_{s}^{-}\right) ^{k}\left(
U_{w_{0}s}U_{w_{0}s}^{-}\right) ^{k}=G_{s}\left(
U_{w_{0}s}U_{w_{0}s}^{-}\right) ^{k}\text{.}
\end{equation*}
Since $L_{s}=G_{s}H$ we can assume $n\in G_{s}$ and hence $nG_{s}=G_{s}n$.
Since $n\in G_{s}\leq L_{s}$, Lemma \ref{min_levi}.(b) gives $n\left(
U_{w_{0}s}U_{w_{0}s}^{-}\right)^{k}=\left( U_{w_{0}s}U_{w_{0}s}^{-}\right)
^{k}n$. Combining everything together yields: 
\begin{eqnarray*}
n \left( UU^{-}\right) ^{k} &=&n G_{s }\left( U_{w_{0}s }U_{w_{0}s
}^{-}\right) ^{k}=G_{s }n \left( U_{w_{0}s }U_{w_{0}s }^{-}\right) ^{k}= \\
&=&G_{s }\left( U_{w_{0}s }U_{w_{0}s }^{-}\right) ^{k}n =\left(
UU^{-}\right) ^{k}n \text{,}
\end{eqnarray*}
and the proof that $N$ normalizes $\left( UU^{-}\right) ^{k}$ is concluded.
\end{proof}

\begin{remark}
If $G$ is generated by $U$ and $U^{-}$ then, in the above proof, Lemma \ref%
{gen_set} can be avoided: if $N$ normalizes $\left( UU^{-}\right) ^{k}$ then 
$\left( UU^{-}\right) ^{k}$ is stable in particular under conjugation by $%
n_{0}H$ so it equals $\left( U^{-}U\right) ^{k}$. It is easy to see that if
this equality holds then $G=(UU^{-})^{k}$.
\end{remark}

\subsection{The case $\left\vert W\right\vert =2$}

\begin{lemma}
\label{Lem_HtildeProperties}Let $G$ be a group with a split $BN$-pair $%
(H,U,N)$ and a Weyl group $W=\left\{ 1,s_{1}\right\} $. Set $\left(
U^{-}\right) ^{\ast }:=U^{-}-\left\{ 1\right\} $. Fix an arbitrary $n_{1}\in
N$ such that $s_{1}=n_{1}H$, and set%
\begin{equation*}
\widetilde{H}:=\left\{ h\in H|\exists u^{-}\in \left( U^{-}\right) ^{\ast
},~Uu^{-}U=Un_{1}hU\right\} \text{.}
\end{equation*}%
Then:

\begin{enumerate}
\item[(a)] $U\left( U^{-}\right) ^{\ast }U=Un_{1}\widetilde{H}U$.

\item[(b)] $UU^{-}UU^{-}=UU^{-}\left( \left\{ 1\right\} \cup n_{1}\widetilde{%
H}n_{1}\widetilde{H}\right) \cup Un_{1}\widetilde{H}$.
\end{enumerate}
\end{lemma}

\begin{proof}
(a) Since $W=\left\{ 1,s_{1}\right\} $ we have 
\begin{equation*}
G=B\cup Bn_{1}B=UH\cup UHn_{1}HU=UH\cup Un_{1}HU\text{,}
\end{equation*}%
where the union on the right is disjoint. By \cite[Proposition 2.5.5(i)]%
{carter1985finite}, $B\cap U^{-}=1$. Hence $\left( U^{-}\right) ^{\ast
}\subseteq Un_{1}HU$. Thus, for every $u^{-}\in \left( U^{-}\right) ^{\ast }$
there exist $h\in H$ such that $Uu^{-}U=Un_{1}hU$. But, by definition, $h\in 
\widetilde{H}$, so this proves $U\left( U^{-}\right) ^{\ast }U\subseteq
Un_{1}\widetilde{H}U$. The reverse inclusion is also clear and hence $%
U\left( U^{-}\right) ^{\ast }U=Un_{1}\widetilde{H}U$.

(b) Note that since each element of $H$ normalizes both $U$ and $U^{-}$, the
set $\widetilde{H}$ commutes with $U$. Also, $w_{0}=s_{1}$ and hence $%
n_{1}Un_{1}^{-1}=U^{-}$ and $n_{1}U^{-}n_{1}^{-1}=U$. Given this and the
relation in (a) we get:%
\begin{gather*}
UU^{-}UU^{-}=U\left( U^{-}\right) ^{\ast }UU^{-}\cup UU^{-}=Un_{1}\widetilde{%
H}UU^{-}\cup UU^{-}= \\
=UU^{-}n_{1}\widetilde{H}U^{-}\cup UU^{-}=UU^{-}Un_{1}\widetilde{H}\cup
UU^{-}= \\
=U\left( U^{-}\right) ^{\ast }Un_{1}\widetilde{H}\cup Un_{1}\widetilde{H}%
\cup UU^{-}= \\
=Un_{1}\widetilde{H}Un_{1}\widetilde{H}\cup Un_{1}\widetilde{H}\cup
UU^{-}=UU^{-}n_{1}\widetilde{H}n_{1}\widetilde{H}\cup Un_{1}\widetilde{H}%
\cup UU^{-}= \\
=UU^{-}\left( \left\{ 1\right\} \cup n_{1}\widetilde{H}n_{1}\widetilde{H}%
\right) \cup Un_{1}\widetilde{H}\text{.}\qedhere
\end{gather*}
\end{proof}

\begin{lemma}
\label{Lem_characterization_rank1} Let $G$ be a group with a split $BN$-pair 
$(H,U,N)$ and Weyl group $W=\left\{ 1,s_{1}\right\} $. Using the notation of
Lemma \ref{Lem_HtildeProperties}, the following conditions are equivalent:

\begin{enumerate}
\item[(a)] $\left( U^{-}\right) ^{\ast }\cap Un_{1}hU\neq \emptyset $ for
all $h\in H$. Equivalently $H=\widetilde{H}$.

\item[(b)] $U\left( U^{-}\right) ^{\ast }U=Un_{1}HU$.

\item[(c)] $G=\left( UU^{-}\right) ^{2}$.
\end{enumerate}
\end{lemma}

\begin{proof}
By definition $\widetilde{H}\subseteq H$ and by Lemma \ref%
{Lem_HtildeProperties} (a), $U\left( U^{-}\right) ^{\ast }U=Un_{1}\widetilde{%
H}U$. Hence (a) and (b) are equivalent. To finish the proof observe that 
\begin{gather*}
G=B\cup Bn_{1}B=\left( B\cup Bn_{1}B\right) n_{1}=Bn_{1}\cup
Bn_{1}Bn_{1}=Bn_{1}\cup BB^{-}= \\
=Un_{1}H\cup UU^{-}H\text{,}
\end{gather*}%
where the union on the r.h.s. is disjoint. Since the sets $\widetilde{H}$
and $n_{1}\widetilde{H}n_{1}$ are both contained in $H$, we have $Un_{1}%
\widetilde{H}\subseteq Un_{1}H$, and $UU^{-}\left( \left\{ 1\right\} \cup
n_{1}\widetilde{H}n_{1}\widetilde{H}\right) \subseteq UU^{-}H$. Since $%
G=Un_{1}H\cup UU^{-}H$ is a disjoint union, Lemma \ref{Lem_HtildeProperties}
(b) implies that $G=\left( UU^{-}\right) ^{2}$ if and only if $Un_{1}%
\widetilde{H}=Un_{1}H$ and $UU^{-}\left( \left\{ 1\right\} \cup n_{1}%
\widetilde{H}n_{1}\widetilde{H}\right) =UU^{-}H$. Thus, by Lemma \ref%
{Lem_HtildeProperties} (a), we get that (c) implies (b), and it is also
clear that (a) implies (c).
\end{proof}

\subsection{Groups with a $\protect\sigma $-setup}

Any $K\in {\mathcal{L}ie}$ has a split $BN$-pair $(H,U,N)$, where $U$ is a
Sylow $p$-subgroup for the defining characteristic $p$, descending from the
algebraic group $\overline{K}$ (see \cite[Theorem 2.3.4]{GLS1998number3}).
More precisely, if $\overline{T}\subseteq \overline{B}$ is a pair of $\sigma 
$-stable maximal torus and Borel subgroup of $\overline{K}$ then $B=%
\overline{B}\cap K$ and $N=N_{\overline{K}}(\overline{T})\cap K$ form a $BN$%
-pair for $K$ and if $\overline{U}$ is the unipotent radical of $\overline{K}
$, i.e., $\overline{B}=\overline{T}\ltimes \overline{U\text{,}}$ then $%
B=H\ltimes U$ where $H=\overline{T}\cap K$ and $U=\overline{U}\cap K$ is a
Sylow $p$-subgroup of $K$ (as in \cite[Section 3.4]{GLS1998number3}).

\begin{remark}
\label{remarks_lie_type}

1.) Some groups in ${\mathcal{L}ie}$ have split $BN$-pairs for different
primes $p$, e.g. $A_{1}(4)=A_{1}(5)$ (see \cite[Theorem 2.2.10]%
{GLS1998number3}).

2.) If $K\in {\mathcal{L}ie}(p)$ then $K_{s}\in {\mathcal{L}ie}(p)$.
Moreover if $K$ is universal then so is $K_{s}$ by \cite[Theorem 2.6.5.(f)]%
{GLS1998number3}.

3.) Note also that if $\overline{K}$ is universal (see \cite[Theorem 1.10.4]%
{GLS1998number3}) then, by a result of Steinberg (see \cite[Theorem 24.15]%
{MaTe2011linear}), $K_{u}=C_{\overline{K}}(\sigma )$ so $B$, $N$, $H$ and $U$
are the centralizers of $\sigma $ in $\overline{B}$, $\overline{N}$, $%
\overline{T}$ and $\overline{U}$ respectively.
\end{remark}

\begin{lemma}
\label{Lem_universal_rank_1} Let $K_{u}\in {\mathcal{L}ie}(p)$ be universal
of rank $1$ and let $U$ be a Sylow $p$-subgroup of $K_{u}$. Then $%
K_{u}=\left( UU^{-}\right) ^{2}$.
\end{lemma}

\begin{proof}
First note that since $K_{u}$ is universal, the corresponding algebraic
group $\overline{K}_{u}$ is universal (or simply connected in a different
terminology \cite[Definition 1.10.5]{GLS1998number3}). By Remark \ref%
{remarks_lie_type}.3, $K_{u}$ is a finite group of Lie type. The possible
types for rank $1$ are $A_{1}$, ${}^{2}A_{2}$, ${}^{2}B_{2}$ and $%
{}^{2}G_{2} $ (see for example \cite[Table 23.1]{MaTe2011linear}) and the
possibilities for $\overline{K}_{u}$ can be read off from \cite[Theorem
1.10.7]{GLS1998number3}.

Let $p$ be the defining characteristic of $K_{u}$. By \cite[\S 1.19]%
{carter1985finite}, we need to consider, for all powers $q$ of $p$, the
groups $\func{SL}_{2}(q)$, $\func{SU}_{3}(q^{2})$, ${}^{2}B_{2}(q^{2})$ if $%
p=2$ and $q^{2}=2^{2n-1}$ for some $n\geq 0$ and ${}^{2}G_{2}(q^{2})$ if $%
p=3 $ and $q^{2}=3^{2n-1}$ for some $n\geq 0$.

Now $K_{u}$ satisfies the assumptions of Lemma \ref%
{Lem_characterization_rank1}, so, in particular we use the notation of Lemma %
\ref{Lem_characterization_rank1}. For $K_{u}=\func{SL}_{2}(q)$ condition (a)
of the lemma is easily verified - for the calculation see \cite[\S 6.1]%
{carter1989simple}. For the remaining cases we use \cite[Proposition 4.1]%
{CEG2000gauss}. By this result, for every $h\in H$ there exists $y\in U$
such that $yn_{1}\in U^{-}hU=n_{1}^{-1}Un_{1}hU$. Multiplying by $n_{1}^{-1}$
on the left, and using $\left( n_{1}^{-1}\right) ^{2}\in H$, we obtain $%
n_{1}^{-1}yn_{1}\in Un_{1}\left( \left( n_{1}^{-1}\right) ^{2}h\right) U$.
Observe that $1\notin Un_{1}\left( \left( n_{1}^{-1}\right) ^{2}h\right) U$,
and hence $n_{1}^{-1}yn_{1}\in \left( U^{-}\right) ^{\ast }$. Moreover, as $%
h $ varies over $H$, so does $\left( n_{1}^{-1}\right) ^{2}h$. Hence,
condition (a) of Lemma \ref{Lem_characterization_rank1} holds for this case,
and the claim follows.
\end{proof}

\begin{remark}
Note that the groups denoted by $\func{SU}_{n}(q^{2})$ in \cite[\S 1.19]%
{carter1985finite} are denoted by $\func{SU}_{n}(q)$ in \cite[Example 21.2]%
{MaTe2011linear}. Note also that for the groups ${}^{2}B_{2}(2^{2n-1})$ the
universal and the adjoint versions are isomorphic (see \cite[\S 1.19]%
{carter1985finite}). Moreover since the center $Z(K_{u})$ lies in $C_{Z(%
\overline{K}_{u})}(\sigma )$ (see \cite[Corollary 24.13]{MaTe2011linear}) it
follows that $Z(K_{u})=1$ except if $K_{u}=\func{SL}_{2}(q)$ and $q$ is odd
(here $Z(K_{u})=Z_{2}$) or $K_{u}=\func{SU}_{3}(q^{2})$ and $3$ divides $q+1$
(here $Z(K_{u})=Z_{3}$). Excluding these exceptions, $K_{u}$ is isomorphic
to its adjoint version, i.e. $K_{u}\cong K_{a}$ and condition (a) of Lemma %
\ref{Lem_characterization_rank1} can be checked with the calculation in \cite%
[\S 13.7]{carter1989simple}.
\end{remark}

The next lemma extends an observation of \cite{VSS2011unitriangular} to the
split $BN$-pair setting.

\begin{lemma}
\label{Lem_HIntesrsectUU_Utrivially}If $G$ is a group with a split $BN$-pair 
$(H,U)$, then $H\cap UU^{-}U=\{1\}$.
\end{lemma}

\begin{proof}
Let $h\in H\cap UU^{-}U$. Then $h\in u_{1}U^{-}u_{2}$, with $u_{1},u_{2}\in
U $. Equivalently, $u_{1}^{-1}hu_{2}^{-1}=h\left( u_{1}^{-1}\right)
^{h}u_{2}^{-1}\in U^{-}$. But $h\left( u_{1}^{-1}\right) ^{h}u_{2}^{-1}\in
B=HU$, and hence $h\left( u_{1}^{-1}\right) ^{h}u_{2}^{-1}\in B\cap
U^{-}=\left\{ 1\right\} $ (\cite[Proposition 2.5.5(i)]{carter1985finite}).
Using $H\cap U=\left\{ 1\right\} $ this gives $h=1$.
\end{proof}

\begin{proof}[\textbf{Proof of Theorem \protect\ref{Th_unipotent_Sylows}}]
We have to show that $G=(UU^{-})^{2}$ for each prime $p$ and each $G\in {%
\mathcal{L}ie}(p)$. Since $G$ satisfies the assumptions of Proposition \ref%
{reduction_rank1}, we can assume that $G$ is in ${\mathcal{L}ie}$ of rank $1$%
. Moreover, since there is a surjective homomorphism $\phi :K_{u}\rightarrow
K$ which maps unipotent Sylows onto unipotent Sylows, we can assume that $G$
is universal. A universal $G$ in ${\mathcal{L}ie}$ of rank $1$ satisfies $%
G=(UU^{-})^{2}$ by Lemma \ref{Lem_universal_rank_1}.

By Lemma \ref{Lem_HIntesrsectUU_Utrivially}, $H\cap UU^{-}U=\{1\}$ and so,
employing Lemma \ref{Lem_reduction_to_UV}, if $H\neq 1$ we must have $\gamma
_{\func{cp}}^{U}\left( G\right) >3$, and hence $G=(UU^{-})^{2}$ is a minimal
length $\func{cp}$-factorization of $G$ by $U$ (i.e., $\gamma _{\func{cp}%
}^{U}\left( G\right) =4$). On the other hand, if $H=1$ then $B=H\ltimes U=U$
and it follows (see \cite[Theorem 5]{CGLMS2014conjugates}) that $G=UU^{-}U$
(i.e., $\gamma _{\func{cp}}^{U}\left( G\right) =3$). Moreover $H=1$ if and
only if $U$ is self-normalizing, i.e., if $U$ is a Carter subgroup.
\end{proof}

\section{Solvable $\func{cp}$-factorizations\label%
{Sect_solvable_factorizations}}

In this section we prove Theorem \ref{Th_loglog}. The proof is based on
reducing the problem of finding an upper bound on $\gamma _{\func{cp}}^{%
\func{s}}(G)$ for a general finite\footnote{%
From here to the end of the paper, all groups are assumed to be finite
unless otherwise stated.} group $G$ to finding an upper bound on the minimal
length of a special kind of a solvable conjugate factorization of a simple
non-abelian group.

\begin{definition}
\label{Def_special_solvable_cp} A $\func{cp}$-factorization $G=A_{1}\cdots
A_{k}$ of a group $G$ by $A\leq G$ will be called a special solvable $\func{%
cp}$-factorization if the following conditions hold:

\begin{enumerate}
\item[(i)] $A$ is solvable.

\item[(ii)] $A$ is self-normalizing in $G$.

\item[(iii)] For any $\alpha \in Aut\left( G\right) $ there exists $g\in G$
such that $A^{\alpha }=A^{g}$.
\end{enumerate}
\end{definition}

The next lemma shows the existence of special solvable $\func{cp}$%
-factorizations for any finite group $G$.

\begin{lemma}
\label{Lem_ExistenceOfSpecialSolFacs}Let $G$ be a group, $p$ a prime, and
let $P$ be a Sylow $p$-subgroup of $G$. Then $A:=N_{G}\left( P\right) $
satisfies properties (ii) and (iii) in Definition \ref%
{Def_special_solvable_cp}, and $G$ is a product of some conjugates of $A$ in 
$G$. If, in addition, $A$ is solvable then this product is a special
solvable conjugate factorization of $G$. Furthermore, if $p=2$ then $A$ is
solvable so any group $G$ has at least one special solvable conjugate
factorization.
\end{lemma}

\begin{proof}
It is well-known that as a consequence of Sylow's theorems, properties (ii)
and (iii) in Definition \ref{Def_gamma_sscp} are satisfied by any Sylow
normalizer subgroup of $G$ (see, for instance, \cite[5.13, 5.14]%
{Rose1978groups}). In order to show that $G$ is a product of conjugates of $%
A $ it suffices to prove that $G=$ $A^{G}$. Observe that $P$ is a Sylow $p$%
-subgroup of $A^{G}$ and clearly $A^{G}\trianglelefteq G$. Hence, by
Frattini's argument, $G=A\left( A^{G}\right) =A^{G}$. Finally, if $p=2$ then 
$A$ is solvable by the Odd Order Theorem.
\end{proof}

\begin{definition}
\label{Def_gamma_sscp}For a finite group $G$ we denote by $\gamma _{\func{cp}%
}^{\func{ss}}(G)$ the minimal length of a special solvable $\func{cp}$%
-factorization of $G$. For a prime $p$ we let $\gamma _{\func{cp}}^{p}(G)$
denote the minimal length of a $\func{cp}$-factorization of $G$ whose
factors are conjugates of a solvable normalizer in $G$ of a Sylow $p$%
-subgroup of $G$, if such a factorization exists, or $\gamma _{\func{cp}%
}^{p}(G)=\infty $ otherwise.
\end{definition}

Note that for every prime $p$ we have $\gamma _{\func{cp}}^{\func{ss}%
}(G)\leq \gamma _{\func{cp}}^{p}(G)$ by Lemma \ref%
{Lem_ExistenceOfSpecialSolFacs}.

\subsection{Special solvable factorizations of simple groups\label%
{SubSection_SpecialSolvable}}

Here we obtain an upper bound on $\gamma _{\func{cp}}^{\func{ss}}(G)$, where 
$S$ is a simple non-abelian group. We discuss separately simple groups of
Lie type, alternating groups and simple sporadic groups, and then combine
the various results in Theorem \ref{Th_gammaSpecialForSimple}.

\begin{lemma}
\label{Lem_ss_Lie} If $S$ is a simple group of Lie type of characteristic $p$
then $\gamma _{\func{cp}}^{\func{ss}}(S)=\gamma _{\func{cp}}^{p}(S)=3$.
\end{lemma}

\begin{proof}
$S$ is a group with a $BN$-pair, where $B$, the Borel subgroup of $S$ is
solvable and is the normalizer of a Sylow $p$-subgroup of $S$. Hence $\gamma
_{\func{cp}}^{p}(S)=3$ by \cite[Theorem 3]{CGLMS2014conjugates}.
\end{proof}

\begin{lemma}
\label{Lem_ss_An} If $S\cong A_{n}$ for $n\geq 5$ then $\gamma _{\func{cp}}^{%
\func{ss}}(S)<12\log _{2}(n)$.
\end{lemma}

\begin{proof}
We first show that the symmetric group $S_{n}$ is a product of less than $%
4\log _{2}(n)$ Sylow $2$-subgroups, adjusting the ideas of the proof of \cite%
[Theorem 2]{Abert2002abelian} to our needs. For any positive integer $n$ set 
$\Omega _{n}:=\left\{ 1,2,...,n\right\} $. Denote the minimal length of a $%
\func{cp}$-factorization of $G$ whose factors are Sylow $2$-subgroups, by $%
f(n)$. First we show that $f(n+1)\leq f(n)+2$. Let $A\cong S_{n}$ be the
point stabilizer of $1$, with respect to the natural action of $S_{n+1}$ on $%
\Omega _{n+1}$. Then $A$ is a product of $f(n)$ Sylow $2$-subgroups of $A$
each of which is a subgroup of a Sylow $2$-subgroup of $S_{n+1}$. Next we
prove that there exist two Sylow $2$-subgroups, $P$ and $Q$ of $S_{n+1}$,
such that $PQ$ contains elements $g_{1}=1_{S_{n+1}},g_{2},...,g_{n+1}$
satisfying $\left( 1\right) g_{i}=i$ for each $1\leq i\leq n+1$ ($\left(
1\right) g_{i}$ stands for the image of $1\in \Omega _{n+1}$ under the
action of $g_{i}\in S_{n+1}$). Note that a subset $\left\{
g_{1},...,g_{n+1}\right\} $ of $S_{n+1}$ whose elements satisfy the last
condition is a right transversal of $A$ in $S_{n+1}$, for if $i\neq j$ then $%
\left( 1\right) g_{i}g_{j}^{-1}\neq 1$, implying $g_{i}g_{j}^{-1}\notin A$.
Clearly, if $PQ$ contains a right transversal of $A$ in $S_{n+1}$, we have $%
APQ=S_{n+1}$, and $f(n+1)\leq f(n)+2$ follows.

Let $k$ be the unique integer satisfying $2^{k}\leq n+1<2^{k+1}$. We can
choose $P$ to be a Sylow $2$-subgroup of $S_{n+1}$ containing $\left\langle
\left( 1,...,2^{k}\right) \right\rangle $, and $Q$ a Sylow $2$-subgroup of $%
S_{n+1}$ containing $\left\langle \left(
n-2^{k}+2,...,n+1\right)\right\rangle $. These two cyclic subgroups act
transitively on their supports, and their supports have at least one point
in common. Hence $PQ$ contains a subset $\left\{ g_{1},...,g_{n+1}\right\} $
having the desired property.

Next we show that if $n$ is even then $f(n)\leq 2f(2)+f(n/2)$. In this case, 
$\Omega _{n}$ is in bijection with the set $\widetilde{\Omega }_{n}:=\left\{
1,2\right\} \times \Omega _{n/2}$. The natural action of $S_{n}$ on $\Omega
_{n}$ induces an action of $S_{n}$ on $\widetilde{\Omega }_{n}$. Let $A$ be
the subgroup consisting of all $g\in S_{n}$ such that for any $\left(
a,b\right) \in \widetilde{\Omega }_{n}$ we have $\left( a,b\right) g=\left(
a,x\right) $ for some $x\in \Omega _{n/2}$ and similarly, $B$ is the
subgroup of $S_{n}$ preserving the second coordinate of the $\widetilde{%
\Omega }_{n}$ element. Then, $A\cong {(S_{n/2})}^{2}$ and $B\cong {(S_{2})}%
^{n/2}$. By \cite[Lemma 4]{Abert2002abelian}, $S_{n}=BAB$. This gives $%
f(n)\leq 2f(2)+f(n/2)$.

Using these two inequalities we prove $f\left( n\right) <4\log _{2}(n)$ by
induction. If $n$ is even then 
\begin{equation*}
f(n)\leq 2+f(n/2)<2+4\log _{2}(n/2)<4\log _{2}(n)\text{.}
\end{equation*}
If $n$ is odd, then 
\begin{equation*}
f(n)\leq 2+f(n-1)\leq 4+f((n-1)/2)<4+4\log _{2}((n-1)/2)<4\log _{2}(n)\text{.%
}
\end{equation*}

Next we show that for $n\geq 6$ the group $A_{n}$ is a product of at most $%
12\log _{2}(n)$ Sylow $2$-subgroups. The group $A_{n}$ acts transitively on $%
P_{2}\left( \Omega _{n}\right) $ the set of all $n(n-1)/2$ subsets of $%
\Omega _{n}$ of size $2$. One can check that the stabilizer of a subset of
size $2$ of $\Omega _{n}$ is isomorphic to $S_{n-2}$. Let $H_{1}$ and $H_{2}$
be the stabilizers of $\{1,2\}$ and $\{n-1,n\}$ respectively. We claim that $%
A_{n}=H_{1}H_{2}H_{1}$. Notice that this claim together with our previous
claim that $S_{n}$ is a product of less than $4\log _{2}(n)$ Sylow $2$%
-subgroups, finish the proof. We have $H_{2}=H_{1}^{g}$ with $g=\left(
1,n-1\right) \left( 2,n\right) $. By \cite[Theorem 1 part 2 (ii)]%
{CGLMS2014conjugates} it is sufficient to show that $\{1,2\}H_{2}$
intersects every $H_{1}$ orbit $O$ on $P_{2}\left( \Omega _{n}\right) $. Let 
$\{i,j\}\in O$ be arbitrary. If $\{i,j\}\cap \{n-1,n\}=\emptyset $, then
there exists an $h\in H_{2}$ so that $\{1,2\}h=\{i,j\}$. On the other hand,
if $\{i,j\}\cap \{n-1,n\}\not=\emptyset $ then, since $n\geq 6$, there
exists $h_{1}\in H_{1}$ so that $\{i,j\}h_{1}\cap \{n-1,n\}=\emptyset $, and
so, since $\{i,j\}h_{1}\in O$ we reduce to the previous case.

Finally, $\gamma _{\func{cp}}^{\func{ss}}(A_{5})=3$ by Lemma \ref{Lem_ss_Lie}
since $A_{5}\cong PSL\left( 2,4\right) $ is of Lie type in characteristic $2$%
. For $n\geq 6$ we have shown that $A_{n}$ is a product of at most $12\log
_{2}(n)$ Sylow $2$-subgroups, hence $\gamma _{\func{cp}}^{2}(A_{n})$ exists
and satisfies $\gamma _{\func{cp}}^{2}(A_{n})<12\log _{2}(n)$. The claim of
the lemma follows.
\end{proof}

\begin{lemma}
\label{Lem_ss_sporadic} If $S$ is a sporadic simple group or the Tits group
then upper bounds on $\gamma _{\func{cp}}^{\func{ss}}(S)$ are given in the
Appendix, in Table \ref{TblSporadics}, under the column heading $\gamma _{%
\func{cp}}^{p}(S)\leq $. It follows that if $S$ is a sporadic simple group
or the Tits group or a simple group of Lie type then $\gamma _{\func{cp}}^{%
\func{ss}}(S)<4.84\log _{2}\log _{2}|S|$.
\end{lemma}

\begin{proof}
The deduction of the upper bounds in Table \ref{TblSporadics} uses several
ingredients. The first one is the detailed information about the maximal
subgroups of the sporadic simple groups which is available in \cite{ATLAS}.
A second ingredient is a basic inequality which relates $\gamma _{\func{cp}%
}^{p}(G)$ to $\gamma _{\func{cp}}^{p}(A)$ for $A\leq G$. This and other
useful relations are summarized in the following lemma.

\begin{lemma}
\label{Lem_ss_SporadicTools}Let $G$ be a group, and $p$ a prime divisor of $%
\left\vert G\right\vert $. Let $P$ be a Sylow $p$-subgroup of $G$.

\begin{enumerate}
\item[(a)] Suppose that $A\leq $ $G$ contains $P$, and that $\gamma _{\func{%
cp}}^{A}(G)$, $\gamma _{\func{cp}}^{p}(G)$ and $\gamma _{\func{cp}}^{p}(A)$
all exist. Then $\gamma _{\func{cp}}^{p}(G)\leq \gamma _{\func{cp}%
}^{A}(G)\gamma _{\func{cp}}^{p}(A)$.

\item[(b)] If $G$ is an almost simple group with socle $S$, $\left\vert
G:S\right\vert $ is not divisible by $p$ and $\gamma _{\func{cp}}^{p}(S)$
exists, then $\gamma _{\func{cp}}^{p}(G)\leq \gamma _{\func{cp}}^{p}(S)$.

\item[(c)] Let $N\trianglelefteq $ $G$. If $N\leq N_{G}\left( P\right) $ is
solvable then $\gamma _{\func{cp}}^{p}(G)=\gamma _{\func{cp}}^{p}(G/N)$.

\item[(d)] If $P$ is non-normal (e.g., $G$ is simple) then $N_{G}\left(
P\right) $ is solvable if and only if for each maximal subgroup $M$ of $G$
such that $\left\vert P\right\vert $ divides $\left\vert M\right\vert $, the
normalizer of a Sylow $p$-subgroup of $M$ is solvable.

\item[(e)] If $G$ is a product of $n$ arbitrary $p$-subgroups, and $%
N_{G}\left( P\right) $ is solvable then $\gamma _{\func{cp}}^{p}(G)\leq n$.

\item[(f)] If $G$ has a normal subgroup $N$ such that $N$ is a product of $n$
arbitrary $p$-subgroups, $G/N$ is a $p$-group, and $N_{G}\left( P\right) $
is solvable, then $\gamma _{\func{cp}}^{p}(G)\leq n$.
\end{enumerate}
\end{lemma}

\begin{proof}
(a) By assumption, $A$ is a product of $\gamma _{\func{cp}}^{p}(A)$ $A$%
-conjugates of $N_{A}\left( P\right) $, and $G$ is a product of $\gamma _{%
\func{cp}}^{A}(G)$ $G$-conjugates of $A$. Hence $G$ is a product of $\gamma
_{\func{cp}}^{A}(G)\gamma _{\func{cp}}^{p}(A)$ $G$-conjugates of $%
N_{A}\left( P\right) $. The claim now follows from $N_{A}\left( P\right)
\leq N_{G}\left( P\right) $.

(b) Since $S\trianglelefteq G$, we have that $S\cap P$ is a Sylow $p$%
-subgroup of $S$, and Since $\left\vert G:S\right\vert $ is not divisible by 
$p$ we get $P=S\cap P$. In \cite[Lemma 14]{GL2014conjugates}, take $m=1$, $%
X=G$ and $U=N_{G}\left( P\right) $. Then $US=N_{G}\left( P\right) S=G$ by
the Frattini's argument, and $U\cap S=$ $N_{S}\left( P\right) $. Since $%
1<N_{S}\left( P\right) <S$ we can deduce from \cite[Lemma 14]%
{GL2014conjugates}, that $G$ is a product of $h=\gamma _{\func{cp}}^{p}(S)$
conjugates of $N_{G}\left( N_{S}\left( P\right) \right) $. Now, because $%
S\trianglelefteq G$, $U=N_{G}\left( P\right) $ normalizes $U\cap S=$ $%
N_{S}\left( P\right) $, and hence $N_{G}\left( P\right) \leq N_{G}\left(
N_{S}\left( P\right) \right) $. Using $N_{G}\left( P\right) S=G$ and
Dedekind's argument we get:%
\begin{gather*}
N_{G}\left( N_{S}\left( P\right) \right) =N_{G}\left( N_{S}\left( P\right)
\right) \cap \left( N_{G}\left( P\right) S\right) =N_{G}\left( P\right)
\left( N_{G}\left( N_{S}\left( P\right) \right) \cap S\right) = \\
N_{G}\left( P\right) N_{S}\left( N_{S}\left( P\right) \right) =N_{G}\left(
P\right) N_{S}\left( P\right) =N_{G}\left( P\right) \text{. }
\end{gather*}%
Therefore $G$ is a product of $\gamma _{\func{cp}}^{p}(S)$ conjugates of $%
N_{G}\left( P\right) $. Finally note that the existence of $\gamma _{\func{cp%
}}^{p}(S)$ implies the solvability of $N_{S}\left( P\right) $ which implies,
using Schreier's Conjecture, the solvability of $N_{G}\left( P\right) $. The
claim follows.

(c) For each $A\leq G$ set $\overline{A}:=AN/N$. Then, in general (without
assuming $N\leq N_{G}\left( P\right) $) we have $\overline{N_{G}\left(
P\right) }=N_{\overline{G}}\left( \overline{P}\right) $ (\cite[3.2.8]%
{KurzUndStell}). If $N\leq N_{G}\left( P\right) $ is solvable then $%
\overline{N_{G}\left( P\right) }=N_{G}\left( P\right) /N=N_{\overline{G}%
}\left( \overline{P}\right) $, and $N_{G}\left( P\right) $ is solvable if
and only if $N_{\overline{G}}\left( \overline{P}\right) $ is solvable.
Moreover, in this case $G$ is a product of $k$ conjugates of $N_{G}\left(
P\right) $ if and only if $\overline{G}$ is a product of $k$ conjugates of $%
N_{\overline{G}}\left( \overline{P}\right) $. The claim follows.

(d) This follows from the fact that $N_{M}\left( P\right) \leq N_{G}\left(
P\right) $ for any $M\leq G$, and from the fact that if $P$ is non-normal,
then $N_{G}\left( P\right) $ is contained in some maximal subgroup of $G$,
so $N_{M}\left( P\right) =N_{G}\left( P\right) $ for some maximal subgroup $%
M $ of $G$, which contains $P$.

(e) Each $p$-subgroup of $G$ is contained in a Sylow $p$-subgroup of $G$
which is, in turn, contained in its normalizer.

(f) By assumption $N=Q_{1}\cdots Q_{n}$ where each $Q_{i}$ is a $p$%
-subgroup. Assume, without loss of generality, $Q_{n}\leq $ $P$. We have $%
G=NP$ because $G/N$ is a $p$-group. Now $G=NP=Q_{1}\cdots Q_{n-1}P$ and the
claim follows from (e).
\end{proof}

A third and crucial ingredient is the possibility to calculate $\gamma _{%
\func{cp}}^{A}(G)$ for many pairs $\left( G,A\right) $ of interest, using
the permutation character $1_{A}^{G}$. If the irreducible decomposition of
this character in terms of the complex irreducible characters of $G$ is
multiplicity free, one can employ a method, developed and implemented in GAP
as a tool called "mfer" by T. Breuer, I. H\"{o}hler and J. M\"{u}ller (\cite%
{BreuerLux}, \cite{BreuerMuller},\cite{GAP} and \cite{GAPCTblLib1.2.1}), in
order to obtain the structure constants of the Hecke algebra of the double
cosets of $A$. From these structure constants one can compute $\gamma _{%
\func{cp}}^{A}(G)$ as explained in \cite[Sections 2.1 and 5]%
{CGLMS2014conjugates}. Note that the "mfer" tool can be applied to groups $G$
where $G$ is a simple sporadic group, as well as to some of the groups
stored in the TomLib library of \cite{GAP}. The fourth ingredient are our
two results from the current paper: Lemma \ref{Lem_ss_Lie} and Theorem \ref%
{Th_unipotent_Sylows}. Further details on how these four ingredients are
used for deducing Table \ref{TblSporadics} are given in the Appendix.

Finally, in order to prove the statement that $\gamma _{\func{cp}}^{\func{ss}%
}(S)<4.84\log _{2}\log _{2}|S|$ holds for all simple non-abelian groups
which are either of Lie type or sporadic or the Tits group, we have to find
the maximum of $u\left( S\right) /\log _{2}\log _{2}|S|$ where $u\left(
S\right) $ is the upper bound we have on $\gamma _{\func{cp}}^{\func{ss}}(S)$%
, and $S$ varies over the groups in question. Let $S$ be a simple group of
Lie type. Then, by Lemma \ref{Lem_ss_Lie}, $u\left( S\right) =3$ and $%
u\left( S\right) /\log _{2}\log _{2}|S|~<1.18$, where the value $1.18$ is
obtained from the minimal value $60$ that $\left\vert S\right\vert $
attains. For the groups in Table \ref{TblSporadics} one gets $u\left(
S\right) /\log _{2}\log _{2}|S|~<4.84$, where the maximum is realized by the 
$HN$ group.
\end{proof}

\begin{theorem}
\label{Th_gammaSpecialForSimple} Let $S$ be a simple non-abelian group. Then 
$\gamma _{\func{cp}}^{\func{ss}}(S)<c\log _{2}\log _{2}|S|$, where $0<c\leq
12$ is a universal constant.
\end{theorem}

\begin{proof}
For $S\cong A_{n}$ assume $n\geq 6$ ($A_{5}$ is treated as a simple group of
Lie type). Using $n!/2\geq \left( n/e\right) ^{n}$ which holds for all $%
n\geq 1$, gives 
\begin{equation*}
\log _{2}(\left\vert S\right\vert )=\log _{2}(n!/2)\geq \log _{2}(\left(
n/e\right) ^{n})=n\log _{2}(n/e)\text{,}
\end{equation*}
and, $\log _{2}\log _{2}(\left\vert S\right\vert )\geq \log _{2}\left(
n\right) +\log _{2}\log _{2}(n/e)$. Since $n\geq 6$, we have $n/e>2$ and $%
\log _{2}\log _{2}(n/e)$ is positive so $\log _{2}\left( n\right) <\log
_{2}\log _{2}(\left\vert S\right\vert )$. Thus, by Lemma \ref{Lem_ss_An}, $%
\gamma _{\func{cp}}^{\func{ss}}(S)<12\log _{2}(n)<12\log _{2}\log
_{2}(\left\vert S\right\vert )$. If $S$ is sporadic or of Lie type we have $%
\gamma _{\func{cp}}^{\func{ss}}(S)\leq 4.84\log _{2}\log _{2}|S|$ by Lemma %
\ref{Lem_ss_sporadic}. Combining the two cases gives the claim of the
theorem.
\end{proof}

\subsection{Reduction to special solvable $\func{cp}$-factorizations}

In this section we reduce the analysis of $\gamma _{\func{cp}}^{\func{s}}$
to that of $\gamma _{\func{cp}}^{\func{ss}}$ for simple groups, and this
reduction enables us to use Theorem \ref{Th_gammaSpecialForSimple} for
proving Theorem \ref{Th_loglog}.

\begin{lemma}
\label{Lem_gamma_sscpDirectProduct} Let $G=T_{1}^{r_{1}}\times\cdots\times
T_{m}^{r_{m}}$ where the $T_{i}$'s are pairwise non-isomorphic non-abelian
simple groups, and $r_{1},...,r_{m}$ are positive integers. Then 
\begin{equation*}
\gamma_{\func{cp}}^{\func{ss}}(G)\leq\max\left\{ \gamma_{\func{cp}}^{\func{ss%
}}(T_{1}),...,\gamma_{\func{cp}}^{\func{ss}}(T_{m})\right\} \text{.}
\end{equation*}
\end{lemma}

\begin{proof}
For each $1\leq i\leq m$ set $k_{i}:=\gamma _{\func{cp}}^{\func{ss}}(T_{i})$%
. By assumption, for each $1\leq i\leq m$ there exists $B_{i1}<T_{i}$
satisfying Definition \ref{Def_special_solvable_cp}, and $T_{i}=B_{i1}\cdots
B_{ik_{i}}$ where $B_{ij}$ is conjugate in $T_{i}$ to $B_{i1}$ for all $%
1\leq j\leq k_{i}$. Set $k:=\max \left\{ k_{1},...,k_{m}\right\} $. We can
assume that the special solvable conjugate factorizations of the $T_{i}$ are
all of equal length $k$, since for each $1\leq i\leq m$ we can add subgroups 
$B_{ij}$ with $k_{i}+1\leq j\leq k$, chosen arbitrarily from the conjugates
of $B_{i1}$ in $T_{i}$. Clearly $T_{i}=B_{i1}\cdots B_{ik}$. We claim that $%
T_{i}^{r_{i}}=B_{i1}^{r_{i}}\cdots B_{ik}^{r_{i}}$ is a special solvable
conjugate factorization of $T_{i}^{r_{i}}$. It is easy to see that each $%
B_{ij}^{r_{i}}$ is solvable, being a direct product of solvable groups, and
that each $B_{ij}^{r_{i}}$ is conjugate in $T_{i}^{r_{i}}$ to $%
B_{i1}^{r_{i}} $ because each $B_{ij}$ is conjugate in $T_{i}$ to $B_{i1}$.
Similarly, $B_{i1}^{r_{i}}$ is self-normalizing in $T_{i}^{r_{i}}$, because $%
B_{i1}$ is self-normalizing in $T_{i}$. In order to verify condition (iii)
of Definition \ref{Def_gamma_sscp}, recall that $Aut\left(
T_{i}^{r_{i}}\right) \cong Aut\left( T_{i}\right) ^{r_{i}}\rtimes S_{r_{i}}$
where the symmetric group $S_{r_{i}}$ permutes the $r_{i}$ direct factors of 
$Aut\left( T_{i}\right) ^{r_{i}}$ according to its natural action on $%
\left\{ 1,...,r_{i}\right\} $ (\cite[9.24]{Rose1978groups}). Thus $S_{r_{i}}$
normalizes $B_{i1}^{r_{i}}$. Let $\alpha \in Aut\left( T_{i}^{r_{i}}\right) $%
. We have $\alpha =g\left( \alpha _{1},...,\alpha _{r_{i}}\right) $ where $%
\alpha _{j}\in Aut\left( T_{i}\right) $, $1\leq j\leq r_{i}$ and $g\in
S_{r_{i}}$. Since $g$ normalizes $B_{i1}^{r_{i}}$ we get $\left(
B_{i1}^{r_{i}}\right) ^{\alpha }=B_{i1}^{\alpha _{1}}\times \cdots \times
B_{i1}^{\alpha _{r_{i}}}$. Now we can use the fact that $B_{i1}$ satisfies
condition (iii) of Definition \ref{Def_gamma_sscp}, as a subgroup of $T_{i}$.

Next define for each $1\leq j\leq k$, $B_{j}:={\textstyle\prod%
\limits_{i=1}^{m}}B_{ij}^{r_{i}}$ (a direct product). We have $G=B_{1}\cdots
B_{k}$, and again we claim that this is a special solvable conjugate
factorization. The proof relies on the previous claim, namely, that $%
T_{i}^{r_{i}}=B_{i1}^{r_{i}}\cdots B_{ik}^{r_{i}}$ is a special solvable
conjugate factorization, and proceeds in the same way where for showing that 
$B_{1}$ satisfies condition (iii) of Definition \ref{Def_gamma_sscp}, we use
the fact that $Aut\left( G\right) =Aut\left( T_{1}^{r_{1}}\right) \times
\cdots \times Aut\left( T_{m}^{r_{m}}\right) $, which follows from the fact
that the $T_{i}$'s are pairwise non-isomorphic non-abelian simple groups (%
\cite[9.25]{Rose1978groups}).

Finally, since $G=B_{1}\cdots B_{k}$ is a special solvable conjugate
factorization, we get $\gamma_{\func{cp}}^{\func{ss}}(G)\leq k$ which is
what we wanted to prove.
\end{proof}

Now we state and prove the main reduction argument.

\begin{lemma}
\label{Lem_gamma(G)<=gamma'(N)+gamma(G/N)} Let $G$ be a finite group and let 
$N\trianglelefteq G$. Then 
\begin{equation*}
\gamma _{\func{cp}}^{\func{s}}(G)\leq \gamma _{\func{cp}}^{\func{ss}%
}(N)+\gamma _{\func{cp}}^{\func{s}}(G/N)\text{.}
\end{equation*}
\end{lemma}

\begin{proof}
Set $t:=\gamma _{\func{cp}}^{\func{s}}(G/N)$. Then, by definition of $\gamma
_{\func{cp}}^{\func{s}}(G/N)$, there exists $H\leq G$ and $N\leq H$ such
that $H/N$ is solvable, and there exist $t$ subgroups $H_{1},\ldots ,H_{t}$
of $G$, all containing $N$, such that $G/N=\left( H_{1}/N\right) \cdots
\left( H_{t}/N\right) $, and $H_{i}/N$ is conjugate to $H/N$ in $G/N$ for
each $1\leq i\leq t$. It follows that $H_{i}$ is conjugate to $H$ in $G$ for
all $1\leq i\leq t$, and $G=H_{1}\cdots H_{t}$.

Set $k:=\gamma _{\func{cp}}^{\func{ss}}(N)$. By definition of $\gamma _{%
\func{cp}}^{\func{ss}}(N)$, there exists $B\leq N$ satisfying (i)-(iii) in
Definition \ref{Def_special_solvable_cp} and $N=B_{1}\cdots B_{k}$, where
each $B_{i}$ is conjugate to $B$ in $N$. We claim that $H=N_{H}\left(
B\right) N$. First note that both $N_{H}\left( B\right) $ and $N$ are
subgroups of $H$ so $N_{H}\left( B\right) N\leq H$. For the reverse
inclusion let $h\in H$ be arbitrary. Since $N\trianglelefteq H$, $h$ acts on 
$N$ as an automorphism, and therefore, by property (iii) in Definition \ref%
{Def_special_solvable_cp}, there exists $n\in N$ such that $B^{h}=B^{n}$
from which it follows that $hn^{-1}\in N_{H}\left( B\right) $. Hence $%
h=\left( hn^{-1}\right) n\in N_{H}\left( B\right) N$.

Now $H=N_{H}\left( B\right) N$ implies that $H/N=N_{H}\left( B\right)
N/N\cong N_{H}\left( B\right) /B$ (by Definition \ref%
{Def_special_solvable_cp}.(ii)). But since both $H/N$ and $B$ are solvable,
we get that $N_{H}\left( B\right) $ is solvable.

For each $1\leq i\leq t$ let $g_{i}\in G$ be such that $H_{i}=H^{g_{i}}$,
and for each $1\leq j\leq k$ let $n_{j}\in N$ be such that $B_{j}=B^{n_{j}}$%
. Using the above we get:%
\begin{gather*}
G=H_{1}\cdots H_{t}=H^{g_{1}}\cdots H^{g_{t}}=\left( N_{H}\left( B\right)
N\right) ^{g_{1}}\cdots\left( N_{H}\left( B\right) N\right) ^{g_{t}}= \\
=\left( N_{H}\left( B\right) \right) ^{g_{1}}\cdots\left( N_{H}\left(
B\right) \right) ^{g_{t}}N=\left( N_{H}\left( B\right) \right)
^{g_{1}}\cdots\left( N_{H}\left( B\right) \right) ^{g_{t}}B_{1}\cdots B_{k}=
\\
=\left( N_{H}\left( B\right) \right) ^{g_{1}}\cdots\left( N_{H}\left(
B\right) \right) ^{g_{t}}B^{n_{1}}\cdots B^{n_{k}}= \\
=\left( N_{H}\left( B\right) \right) ^{g_{1}}\cdots\left( N_{H}\left(
B\right) \right) ^{g_{t}}\left( N_{H}\left( B\right) \right)
^{n_{1}}\cdots\left( N_{H}\left( B\right) \right) ^{n_{k}}\text{.}
\end{gather*}
Since $N_{H}\left( B\right) $ is solvable this implies that $\gamma_{\func{cp%
}}^{\func{s}}(G)\leq k+t$ as claimed.
\end{proof}

Our next definition is required for the application of Lemma \ref%
{Lem_gamma(G)<=gamma'(N)+gamma(G/N)}. We denote by $R\left( G\right) $ the
solvable radical of $G$ and by $\func{soc}\left(G\right) $ the socle of $G$.

\begin{definition}
\label{Def_nab_socle_series} Let $G$ be a finite group. \emph{The
non-abelian socle series} of $G$ is the unique normal series $R\left(
G\right) =H_{1}\leq \ldots \leq H_{t}=G$ of $G$ which satisfies the
following conditions:

\begin{enumerate}
\item[(i)] for all $1\leq i\leq (t-1)/2$ we have $H_{2i+1}/H_{2i}=R\left(
G/H_{2i}\right) $,

\item[(ii)] for all $1\leq i\leq t/2$ we have $H_{2i}/H_{2i-1}=\func{soc}%
\left( G/H_{2i-1}\right) $.
\end{enumerate}

The number $\left\lfloor t/2\right\rfloor $ will be called \emph{the
non-abelian socle length} of $G$.
\end{definition}

In the sequel we will denote $N_{i}:=H_{2i}/H_{2i-1}=\func{soc}\left(
G/H_{2i-1}\right) $, and $n_{i}$ will stand for the number of simple
non-abelian direct factors of $N_{i}$ for all $1\leq i\leq t/2$. Observe
that the uniqueness of the non-abelian socle series of $G$ is a consequence
of the uniqueness of the solvable radical and the socle of any given finite
group. Moreover, for all $1\leq i\leq t/2$ we have $R\left(
G/H_{2i-1}\right) =1$. This is clear for $i=1$, and for $i\geq 2$ we have $%
G/H_{2i-1}\cong \left( G/H_{2i-2}\right) /\left( H_{2i-1}/H_{2i-2}\right) $
and now we can use $H_{2i-1}/H_{2i-2}=R\left( G/H_{2i-2}\right) $. Since $%
R\left( G/H_{2i-1}\right) =1$ we get that $N_{i}$ is a non-trivial direct
product of non-abelian simple groups. As a result, the inclusion $%
H_{2i-1}\leq H_{2i}$ is always strict, while the inclusion $H_{2i}\leq
H_{2i+1}$ need not be strict. Finally note that the non-abelian socle length
of $G$ is zero if and only if $G$ is solvable.

\begin{corollary}
\label{Coro_gamma<= 1+nabSocleLength}Let $G$ be a non-trivial finite group
whose non-abelian socle length is $m\geq0$. For each $1\leq i\leq m$ pick a
simple non-abelian direct factor $T_{i}$ of $N_{i}$ such that $\gamma_{\func{%
cp}}^{\func{ss}}(T_{i})$ is maximal compared to any other factor of $N_{i}$.
Then 
\begin{equation}
\gamma_{\func{cp}}^{\func{s}}(G)\leq1+{\textstyle\sum\limits_{i=1}^{m}}
\gamma_{\func{cp}}^{\func{ss}}(T_{i})\text{.}  \label{Ineq_gamma_scp}
\end{equation}
\end{corollary}

\begin{proof}
By induction on $m\geq 0$. If $m=0$ then $G$ is solvable and so $\gamma _{%
\func{cp}}^{\func{s}}(G)=1$. Suppose $m>0$. Then $\gamma _{\func{cp}}^{\func{%
s}}(G)\leq \gamma _{\func{cp}}^{\func{s}}(G/R\left( G\right) )$. In fact
equality holds since if $G=A_{1}\dots A_{k}$ is a solvable $\func{cp}$%
-factorization then so is $G=\left( A_{1}R\left( G\right) \right) \dots
\left( A_{k}R\left( G\right) \right) $. Moreover, if $R\left( G\right)
=H_{1}\leq \ldots \leq H_{t}=G$ is the non-abelian socle series of $G$, then 
$1=H_{1}/R\left( G\right) \leq \ldots \leq H_{t}/R\left( G\right) =G/R\left(
G\right) $ is the non-abelian socle series of $G/R\left( G\right) $. We have 
$R\left( G/R\left( G\right) \right) =1$ and $N_{1}=H_{2}/R\left( G\right) $.
Hence, by Lemma \ref{Lem_gamma(G)<=gamma'(N)+gamma(G/N)} we have $\gamma _{%
\func{cp}}^{\func{s}}(G)=\gamma _{\func{cp}}^{\func{s}}(G/R\left( G\right)
)\leq \gamma _{\func{cp}}^{\func{ss}}(N_{1})+\gamma _{\func{cp}}^{\func{s}%
}(\left( G/R\left( G\right) \right) /\left( H_{2}/R\left( G\right) \right) )$%
. Using $\left( G/R\left( G\right) \right) /\left( H_{2}/R\left( G\right)
\right) \cong G/H_{2}$, we obtain $\gamma _{\func{cp}}^{\func{s}}(G)\leq
\gamma _{\func{cp}}^{\func{ss}}(N_{1})+\gamma _{\func{cp}}^{\func{s}}\left(
G/H_{2}\right) $. By Lemma \ref{Lem_gamma_sscpDirectProduct}, $\gamma _{%
\func{cp}}^{\func{ss}}(N_{1})\leq \gamma _{\func{cp}}^{\func{ss}}(T_{1})$,
and since the non-abelian socle length of $G/H_{2}$ is $m-1$, we have by
induction $\gamma _{\func{cp}}^{\func{s}}(G/H_{2})\leq 1+{\textstyle%
\sum\limits_{i=2}^{m}}\gamma _{\func{cp}}^{\func{ss}}(T_{i})$. The claim
follows.
\end{proof}

Theorem \ref{Th_gammaSpecialForSimple} provides an upper bound on each term
in the sum $\sum_{i=1}^{m}\gamma _{\func{cp}}^{\func{ss}}(T_{i})$ appearing
on the r.h.s. of Inequality \ref{Ineq_gamma_scp}. The last ingredient of the
proof of Theorem \ref{Th_loglog} is an upper bound on the number of terms in
this sum which is equal to the non-abelian socle length of $G$. Recall that
for any finite group $H$ there exists the least integer $n$, customarily
denoted $\mu (H)$, so that $H$ embeds in the symmetric group $S_{n}$. In
other words, $\mu \left( H\right) $ is the minimal degree of a faithful
permutation representation of $H$. We will make use of the following
properties of this quantity. If $H_{1}\leq H$ then $\mu \left( H_{1}\right)
\leq \mu \left( H\right) $ (immediate from the definition). If $%
N=T_{1}\times \cdots \times T_{k}$ where $T_{i}$ simple non-abelian for each 
$1\leq i\leq k$ then $\mu \left( N\right) ={\textstyle\sum\limits_{i=1}^{k}}%
\mu \left( T_{i}\right) $ (\cite[Theorem 3.1]{KP1988permutation_rep}). If $%
N\trianglelefteq H$ and $R\left( H/N\right) =1$ then $\mu \left( H/N\right)
\leq \mu \left( H\right) $ (\cite[Theorem 1]{KP2000permutation_rep}). If $T$
is simple non-abelian then $\mu \left( T\right) \geq 5$ (since all subgroups
of $S_{n}$ are solvable if $n<5$).

\begin{lemma}
\label{Lem_mu(G/H_2i-1)<=n_i-1}Let $G$ be a non-solvable group and let $%
R\left( G\right) =H_{1}\leq \ldots \leq H_{t}=G$ be the non-abelian socle
series of $G$. Then, for each $i>1$, $\mu \left( G/H_{2i-1}\right) \leq
n_{i-1}$.
\end{lemma}

\begin{proof}
It is clearly sufficient to prove this for the case $i=2$. So we will prove $%
\mu \left( G/H_{3}\right) \leq n_{1}$. Since $G/H_{1}$ has a trivial
solvable radical it acts faithfully by conjugation on $N_{1}=soc\left(
G/H_{1}\right) =H_{2}/H_{1}$ and so embeds in $Aut\left( N_{1}\right) $. Now 
$N_{1}=T_{1}^{r_{1}}\times \cdots \times T_{m}^{r_{m}}$ where the $T_{i}$'s
are pairwise non-isomorphic non-abelian simple groups, and $r_{1},...,r_{m}$
are positive integers ($\sum_{i=1}^{m}r_{i}=n_{1}$). We have $Aut\left(
N_{1}\right) =Aut\left( T_{1}^{r_{1}}\right) \times \cdots \times Aut\left(
T_{m}^{r_{m}}\right) $ and $Aut\left( T_{i}^{r_{i}}\right) \cong Aut\left(
T_{i}\right) ^{r_{i}}\rtimes S_{r_{i}}$ where the symmetric group $S_{r_{i}}$
permutes the $r_{i}$ direct factors of $Aut\left( T_{i}\right) ^{r_{i}}$
according to its natural action on $\left\{ 1,...,r_{i}\right\} $ (see \cite[%
9.25]{Rose1978groups}). Now, the image of $G/H_{1}$ in $Aut\left(
N_{1}\right) $ contains $Inn\left( T_{1}^{r_{1}}\right) \times \cdots \times
Inn\left( T_{m}^{r_{m}}\right) =\left( Inn\left( T_{1}\right) \right)
^{r_{1}}\times \cdots \times \left( Inn\left( T_{m}\right) \right) ^{r_{m}}$
which is, in fact, the image of $N_{1}$ so 
\begin{gather*}
\left( G/H_{1}\right) /N_{1}=\left( G/H_{1}\right) /\left(
H_{2}/H_{1}\right) \cong G/H_{2}\precsim Aut\left( N_{1}\right)
/\prod_{i=1}^{m}\left( Inn\left( T_{i}\right) \right) ^{r_{i}} \\
\cong \prod_{i=1}^{m}Out\left( T_{i}\right) ^{r_{i}}\rtimes S_{r_{i}}\text{,}
\end{gather*}%
where $\prod_{i=1}^{m}$ is direct, $Out\left( T_{i}\right) :=Aut\left(
T_{i}\right) /Inn\left( T_{i}\right) $, and $\precsim $ denotes embedding.
Set $A:=\prod_{i=1}^{m}Out\left( T_{i}\right) ^{r_{i}}\rtimes S_{r_{i}}$, $%
B:=\prod_{i=1}^{m}Out\left( T_{i}\right) ^{r_{i}}$ and $S:=%
\prod_{i=1}^{m}S_{r_{i}}\leq S_{n_{1}}$. We have $A=BS$, and $%
B\trianglelefteq A$. Furthermore, $B$ is solvable by Schreier's conjecture.
Hence $B\leq R\left( A\right) $. Therefore $A/R\left( A\right) =SR\left(
A\right) /R\left( A\right) \cong S/\left( S\cap R\left( A\right) \right) $.
Since $R\left( A/R\left( A\right) \right) =1$ we have $R\left( S/\left(
S\cap R\left( A\right) \right) \right) =1$. Thus, by \cite[Theorem 1]%
{KP2000permutation_rep}, 
\begin{equation*}
\mu \left( A/R\left( A\right) \right) =\mu \left( S/\left( S\cap R\left(
A\right) \right) \right) \leq \mu \left( S\right) \leq \mu \left(
S_{n_{1}}\right) =n_{1}\text{.}
\end{equation*}%
Now, identifying $G/H_{2}$ with its embedding in $A$, we have that $\left(
G/H_{2}\right) R\left( A\right) /R\left( A\right) $ is a subgroup of $%
A/R\left( A\right) $ and so $\mu \left( \left( G/H_{2}\right) R\left(
A\right) /R\left( A\right) \right) \leq \mu \left( A/R\left( A\right)
\right) $. On the other hand, 
\begin{equation*}
\left( G/H_{2}\right) R\left( A\right) /R\left( A\right) \cong \left(
G/H_{2}\right) /\left( \left( G/H_{2}\right) \cap R\left( A\right) \right) 
\text{.}
\end{equation*}%
Set $D:=\left( G/H_{2}\right) \cap R\left( A\right) \leq R\left(
G/H_{2}\right) $. By an isomorphism theorem we have $\left( \left(
G/H_{2}\right) /D\right) /\left( R\left( G/H_{2}\right) /D\right) \cong
\left( G/H_{2}\right) /R\left( G/H_{2}\right) $. Hence, by \cite[Theorem 1]%
{KP2000permutation_rep},:%
\begin{gather*}
\mu \left( \left( G/H_{2}\right) /R\left( G/H_{2}\right) \right) \leq \mu
\left( \left( G/H_{2}\right) /D\right) =\mu \left( \left( G/H_{2}\right)
R\left( A\right) /R\left( A\right) \right) \\
\leq \mu \left( A/R\left( A\right) \right) \leq n_{1}\text{.}
\end{gather*}%
On the other hand 
\begin{equation*}
\left( G/H_{2}\right) /R\left( G/H_{2}\right) =\left( G/H_{2}\right) /\left(
H_{3}/H_{2}\right) \cong G/H_{3}\text{,}
\end{equation*}%
and the claim $\mu \left( G/H_{3}\right) \leq n_{1}$ follows.
\end{proof}

\begin{lemma}
Let $G$ be a non-solvable group, and let $m$ be the non-abelian socle length
of $G$. Then $5n_{i}\leq n_{i-1}$ for all $2\leq i\,\leq m$.
\end{lemma}

\begin{proof}
Let $R\left( G\right) =H_{1}\leq \ldots \leq H_{t}=G$ be the non-abelian
socle series of $G$. By the preceding remarks, $N_{i}=soc\left(
G/H_{2i-1}\right) =T_{i1}\times \cdots \times T_{in_{i}}$ for all $1\leq
i\leq m$, where each $T_{ij}$ is a non-abelian simple group. We have: 
\begin{equation*}
\mu \left( N_{i}\right) =\sum_{j=1}^{n_{i}}\mu \left( T_{ij}\right) \geq
5n_{i}\text{.}
\end{equation*}%
On the other hand, by Lemma \ref{Lem_mu(G/H_2i-1)<=n_i-1}, $\mu \left(
N_{i}\right) \leq \mu \left( G/H_{2i-1}\right) \leq n_{i-1}$. Thus $%
5n_{i}\leq n_{i-1}$ for all $2\leq i\,\leq m$.
\end{proof}

\begin{corollary}
\label{Coro_BoundOnm} Let $G$ be a non-solvable group, and let $m$ be the
non-abelian socle length of $G$. Then $m<(1/\log _{2}5)\log _{2}\log
_{2}\left\vert G\right\vert _{\func{nab}}$.
\end{corollary}

\begin{proof}
By the previous lemma $5n_{i}\leq n_{i-1}$ for all $2\leq i\,\leq m$. Since $%
n_{m}\geq 1$ we get by induction, $n_{i}\geq 5^{m-i}$ for all $1\leq i\,\leq
m$, and hence the total number of non-abelian composition factors of $G$
satisfies 
\begin{equation*}
\sum_{i=1}^{m}n_{i}\geq \sum_{i=1}^{m}5^{m-i}=\sum\limits_{i=0}^{m-1}5^{i}=%
\frac{5^{m}-1}{4}\text{.}
\end{equation*}%
Each non-abelian composition factor of $G$ is of size at least $\left\vert
A_{5}\right\vert =60>2^{5}$ so 
\begin{equation*}
\left\vert G\right\vert _{\func{nab}}\geq 60^{\frac{5^{m}-1}{4}}>2^{\frac{5}{%
4}\left( 5^{m}-1\right) }\geq 2^{5^{m}}\text{,}
\end{equation*}%
which gives $m<(1/\log _{2}5)\log _{2}\log _{2}\left\vert G\right\vert _{%
\func{nab}}$.
\end{proof}

\begin{proof}[\textbf{Proof of Theorem \protect\ref{Th_loglog}}]
By Corollary \ref{Coro_gamma<= 1+nabSocleLength} we have $\gamma _{\func{cp}%
}^{\func{s}}(G)\leq 1+\sum_{i=1}^{m}\gamma _{\func{cp}}^{\func{ss}}(T_{i})$.
Since $T_{i}$ is a non-abelian composition factor we have $\left\vert
T_{i}\right\vert \leq \left\vert G\right\vert _{\func{nab}}$. Hence, by
Theorem \ref{Th_gammaSpecialForSimple}, $\gamma _{\func{cp}}^{\func{ss}%
}(T_{i})<c\log _{2}\log _{2}|G|_{\func{nab}}$ for all $1\leq i\leq m$.
Substituting in the previous inequality we get $\gamma _{\func{cp}}^{\func{s}%
}(G)\leq 1+mc\log _{2}\log _{2}|G|_{\func{nab}}$. Finally, $c\leq 12$ by
Theorem \ref{Th_gammaSpecialForSimple}, and $m<(1/\log _{2}5)\log _{2}\log
_{2}\left\vert G\right\vert _{\func{nab}}$ by Corollary \ref{Coro_BoundOnm},
so the claim of the theorem holds with $c_{S}\leq 12/\log _{2}5$.
\end{proof}

\section{Nilpotent $\func{cp}$-factorizations\label%
{Sect_nilpotent_factorizations}}

In this section we prove Theorem \ref{Th_Carter}. The proof relies on
Theorem \ref{Th_AffinePrimitive} which is of independent interest. Hence we
begin with the latter.

\subsection{$\func{cp}$-factorizations of affine primitive groups\label%
{SubSect_affine_primitive}}

Recall that a group $G$ is said to be primitive if it admits a maximal
subgroup $H$ which is core-free: $H_{G}=\bigcap_{g\in G}H^{g}=\{1\}$. If $G$
is an affine primitive permutation group, then it has exactly one minimal
normal subgroup $V$, which is abelian so $V\cong {C_{p}^{n}}$ for some prime 
$p$ and some natural number $n$. Moreover $G=VH$ and, viewing $V$ as the
vector space over ${\mathbb{F}_{p}}$, then $H$ acts by conjugation
irreducibly as a group of linear transformations on $V$. When convenient we
will use additive notation for $V$.

\begin{lemma}
\label{Lem_trick} Let $G$ be an affine primitive permutation group with
point stabilizer $H$ and minimal normal subgroup $V\cong {C_{p}^{n}}$. Let $%
h\in H$ and $v\in V$. Set $w:=v^{h^{-1}}-v$ and $k:=\lceil \log _{2}p\rceil $%
. Then $\left\langle w\right\rangle $ is contained in a product of $k+1$
conjugates of $H$.
\end{lemma}

\begin{proof}
We can assume $w\neq 0_{V}=1_{G}$ for which the claim clearly holds. Then $w$
is of order $p$, and any element of $\left\langle w\right\rangle $ is of the
additive form $sw$ where the integer $s$ satisfies $0\leq s\leq p-1$. Since $%
k:=\lceil \log _{2}p\rceil $, the base 2 representation of $s$ takes the
form $s=\tsum\limits_{j=0}^{k-1}b_{j}2^{j}$ ($b_{j}\in \left\{ 0,1\right\} $
for all $0\leq j\leq k-1$). Now note that $w=v^{h^{-1}}-v=v^{-1}hvh^{-1}\in
H^{v}H$. Similarly, for any $c\in {\mathbb{F}_{p}}$ we have $cw=\left(
cv\right) ^{h^{-1}}-cv\in H^{cv}H$. Thus, identifying the powers $2^{j}$
with elements of ${\mathbb{F}_{p}}$, we see that $sw\in $ $\left(
H^{v}H\right) \left( H^{2v}H\right) \left( H^{2^{2}v}H\right) \cdots \left(
H^{2^{k-1}v}H\right) $, for any $0\leq s\leq p-1$, where we pick $0_{V}$
from the $j$-th factor $\left( H^{2^{j}v}H\right) $ in the product if $%
b_{j}=0$ and $2^{j}w$ if $b_{j}=1$. However, also note that since $V$ is
abelian, $\left( H^{2^{j}v}H\right) \cap V$ is invariant under conjugation
by any element of $V$. Hence, for any choice of $u_{0},...,u_{k-2}\in V$ we
have 
\begin{equation*}
sw\in \Pi _{H}:=\left( H^{v}H\right) ^{u_{0}}\left( H^{2v}H\right)
^{u_{1}}\left( H^{2^{2}v}H\right) ^{u_{2}}\cdots \left( H^{2^{k-2}v}H\right)
^{u_{k-2}}\left( H^{2^{k-1}v}H\right) \text{.}
\end{equation*}
Finally, for the choice $u_{k-2}=2^{k-1}v$, $u_{k-3}=u_{k-2}+2^{k-2}v$ and
in general $u_{k-j}=u_{k-j+1}+2^{k-j+1}v$ for all $2\leq j\leq k$ where $%
u_{k-1}:=0_{V}$,\ we get that $\Pi _{H}$ is equal to a product of $k+1$
conjugates of $H$.
\end{proof}

\begin{lemma}
\label{Lem_maxf(p)}For each prime number $p$ define $f(p):=\lceil \log
_{2}p\rceil /\log _{2}p$. Then $f(p)$ has a global maximum at $p=5$.
Consequently 
\begin{equation}
\lceil \log _{2}p\rceil \leq (3/\log _{2}5)\log _{2}p\text{, for every prime 
}p\text{.}  \label{Ineq_log2p}
\end{equation}
\end{lemma}

\begin{proof}
First check that $1+1/\log _{2}11<1.29<3/\log _{2}5$. Then, using this, we
get: 
\begin{equation*}
f(p)\leq (\log _{2}p+1)/\log _{2}p=1+1/\log _{2}p<3/\log _{2}5=f(5)\text{, }%
\forall p\geq 11\text{,}
\end{equation*}%
and for $p=2,3,7$ we verify explicitly that $f\left( p\right) <f\left(
5\right) $. Hence $f(p)$ has a global maximum $f\left( 5\right) =3/\log
_{2}5 $ at $p=5$. Finally, $\lceil \log _{2}p\rceil =f(p)\log _{2}p$ $\leq
f(5)\log _{2}p$.
\end{proof}

\begin{proof}[\textbf{Proof of Theorem \protect\ref{Th_AffinePrimitive}}]
Using the notation introduced at the beginning of Subsection \ref%
{SubSect_affine_primitive}, $\log _{2}|G:H|~=$ $\log _{2}|V|~=$ $\log
_{2}p^{n}=n\log _{2}p$. Using Inequality \ref{Ineq_log2p}, we obtain:%
\begin{equation*}
1+n\lceil \log _{2}p\rceil \leq 1+(3/\log _{2}5)n\log _{2}p=1+(3/\log
_{2}5)\log _{2}|G:H|\text{.}
\end{equation*}

Thus, it is enough to show that $G$ is a product of at most $1+$ $n\lceil
\log _{2}p\rceil $ conjugates of $H$.

Fix a non-zero vector $v\in V$. If $v$ is central in $G$ then $V=\langle
v\rangle $ by minimality of $V$. It follows that $H$ is a non-trivial normal
subgroup of $HV=G$ since $V$ is central - a contradiction to $H$ being
core-free. Therefore $v$ is not central, and there is some $h\in H$ with $%
v^{h^{-1}}\neq v$. Set $w:=v^{h^{-1}}-v$.

We claim that there are $n$ elements $h_{1},\ldots ,h_{n}\in H$ such that $%
B:=\{w^{h_{1}},\ldots ,w^{h_{n}}\}$ is a vector space basis of $V=C_{p}^{n}$%
. Note that since $w\neq 0_{V}$, this claim is immediate for $n=1$, and
hence we assume $n\geq 2$. Suppose by contradiction that $1\leq m<n$ is the
maximal integer such that there exist $h_{1},h_{2},\ldots ,h_{m}\in H$ for
which $B=\{w^{h_{1}},\ldots ,w^{h_{m}}\}$ is linearly independent. It
follows that for any $h\in H$, $w^{h}\in Span\left( B\right) $. Thus $%
Span\left( B\right) =Span\left( \left\{ w^{h}|h\in H\right\} \right) $. This
shows that $Span\left( B\right) $ is a proper non-trivial $H$-invariant
subspace of $V$, contradicting the fact that $H$ acts irreducibly on $V$.
Thus there exists a basis of $V$ of the form $B:=\{w^{h_{1}},\ldots
,w^{h_{n}}\}$.

For each $v\in V$ there exist $s_{1},...,s_{n}\in {\mathbb{F}_{p}}$ for
which $v=\Sigma _{i=1}^{n}s_{i}w^{h_{i}}$. Applying Lemma \ref{Lem_trick} to
each $w^{h_{i}}$ separately, we get that each $v\in V$ belongs to $\Pi
_{1}\cdots \Pi _{n}$, where each $\Pi _{i}$ is a product of $\lceil \log
_{2}p\rceil +1$ conjugates of $H$. But, as in the proof of Lemma \ref%
{Lem_trick}, this shows that $V\subseteq \Pi _{1}^{u_{1}}\cdots \Pi
_{n-1}^{u_{n-1}}\Pi _{n}$ for any choice of $u_{1},...,u_{n-1}\in V$, and
one can choose these elements so that the product $\Pi _{1}^{u_{1}}\cdots
\Pi _{n-1}^{u_{n-1}}\Pi _{n}$ is a product of at most $n\lceil \log
_{2}p\rceil +1$ conjugates of $H$.
\end{proof}

\subsection{$\func{cp}$-factorizations of solvable groups by Carter subgroups%
}

Recall that $C\leq G$ is a Carter subgroup of $G$ if $C$ is nilpotent and
self normalizing.

\begin{lemma}
\label{Lem_Carter_properties} Let $G$ be a solvable group. Then

\begin{enumerate}
\item[(a)] There exists a Carter subgroup of $G$.

\item[(b)] There is a unique conjugacy class of Carter subgroups in $G$.

\item[(c)] If $C$ is a Carter subgroup of $G$ then $C$ is a maximal
nilpotent subgroup of $G$, that is, if $C<H\leq G$, then $H$ is not
nilpotent.

\item[(d)] If $C$ is a Carter subgroup of $G$ and $N\trianglelefteq G$ then $%
CN/N$ is a Carter subgroup of $G/N$.

\item[(e)] If $C$ is a Carter subgroup of $G$ then $C$ is Carter subgroup of 
$H$ for any $C\leq H\leq G$.
\end{enumerate}
\end{lemma}

\begin{proof}
For (a)-(d) see \cite{Carter1961carter}, and \cite[Theorem 12.2(b) and Lemma
12.3]{Huppert1967gruppen}. For (e) note that since $C$ is self-normalizing
in $G$ it is self-normalizing in any subgroup of $G$ containing $C$.
\end{proof}

\begin{lemma}
Let $G$ be a solvable group and let $C$ be a Carter subgroup of $G$. Then $G$
is a product of conjugates of $C$.
\end{lemma}

\begin{proof}
It is enough to show that $C^{G}=G$. Set $N:=C^{G}$. Note that $C$ is a
Carter subgroup of $N$. For any $g\in G$ we have $C^{g}\in N$ is a self
normalizing nilpotent subgroup of $N$. Hence $C^{g}$ is a Carter subgroup of 
$N$ and hence there exists $n\in N$ such that $C^{g}=C^{n}$. It follows that 
$gn^{-1}$ normalizes $C$ and therefore $gn^{-1}\in C\leq N$. Hence $g\in N$
implying $G=N$.
\end{proof}

\begin{definition}
\label{Def_gamma_Carter}Let $G$ be a finite solvable group. We denote by $%
\gamma _{\func{cp}}^{\func{c}}\left( G\right) $ the minimal length of a $%
\func{cp}$-factorization of $G$ by a Carter subgroup.
\end{definition}

Note that a nilpotent group $G$ is equal to its own Carter subgroup and
hence, for $G$ nilpotent, $\gamma _{\func{cp}}^{\func{c}}\left( G\right) =1$%
. Clearly $\gamma _{\func{cp}}^{\func{n}}\left( G\right) \leq \gamma _{\func{%
cp}}^{\func{c}}\left( G\right) $.

\begin{lemma}
\label{Lem_NContainedInC}Let $G$ be a solvable group and let $C$ be a Carter
subgroup of $G$. Let $N\trianglelefteq G$ be such that $N$ is contained in $%
C $. Then $\gamma_{\func{cp}}^{\func{c}}\left( G\right) =\gamma_{\func{cp}}^{%
\func{c}}\left(G/N\right) $.
\end{lemma}

\begin{proof}
It is clear that if $N$ is contained in $C$ then it is contained in every
conjugate of $C$. Suppose that $G=C_{1}\cdots C_{k}$ where $C_{i}$ is a
conjugate of $C$ for all $1\leq i\leq k$. Then $\overline{G}=\overline{C_{1}}%
\cdots\overline{C_{k}}$, where, for any $A\leq G$ we denote $\overline{A}%
:=AN/N$, and each $\overline{C_{i}}$ is a Carter subgroup of $\overline{G}$.
Conversely, if $\overline{G}=\overline{C_{1}}\cdots\overline{C_{k}}$, where
the $\overline{C_{i}}$ are Carter subgroups of $\overline{G}$, then, by
assumption, the full preimage of $\overline{C_{i}}$ in $G$ is a Carter
subgroup $C_{i}$ of $G$ and we can conclude that $G=C_{1}\cdots C_{k}$. The
claim follows.
\end{proof}

\begin{proof}[\textbf{Proof of Theorem \protect\ref{Th_Carter}}]
The proof is by induction on $\left\vert G\right\vert $. Let $C$ be a Carter
subgroup of $G$. If $G$ is nilpotent then $G=C$, $\gamma _{\func{cp}}^{\func{%
c}}\left( G\right) =1$, and the claim clearly holds. Hence we can assume
that $G$ is non-nilpotent. Let $N$ be a minimal normal subgroup of $G$. For
any $A\leq G$ denote $\overline{A}:=AN/N$. Then $\overline{G}=\overline{C_{1}%
}\cdots \overline{C_{k}}$, where each $\overline{C_{i}}$ is a Carter
subgroup of $\overline{G}$, and $k=\gamma _{\func{cp}}^{\func{c}}\left( 
\overline{G}\right) $. By Lemma \ref{Lem_Carter_properties}(d), the full
preimage of $\overline{C_{i}}$ in $G$ is $C_{i}N$ where $C_{i}$ is a Carter
subgroup of $G$, and we get $G=C_{1}\cdots C_{k}N$. If $k>1$, $C_{k}N$ is
proper in $G$. By Lemma \ref{Lem_Carter_properties}(e) $C_{k}$ is a Carter
subgroup of $C_{k}N$ and hence we get by induction that $C_{k}N$ is a
product of $\gamma _{\func{cp}}^{\func{c}}\left( C_{k}N\right) \leq
c_{A}\log _{2}\left( \left\vert C_{k}N:C_{k}\right\vert \right) +1$
conjugates of $C_{k}$. Since $C_{k}$ is conjugate to $C$ and $N$ is normal
we have $\left\vert C_{k}N:C_{k}\right\vert =\left\vert CN:C\right\vert $.
Therefore 
\begin{align*}
\gamma _{\func{cp}}^{\func{c}}\left( G\right) & \leq k-1+c_{A}\log
_{2}\left( \left\vert C_{k}N:C_{k}\right\vert \right) +1 \\
& =\gamma _{\func{cp}}^{\func{c}}\left( \overline{G}\right) +c_{A}\log
_{2}\left( \left\vert CN:C\right\vert \right) \\
& \leq c_{A}\log _{2}\left( \left\vert G/N:CN/N\right\vert \right)
+1+c_{A}\log _{2}\left( \left\vert CN:C\right\vert \right) \\
& =c_{A}\log _{2}\left( \left\vert G:C\right\vert \right) +1\text{,}
\end{align*}%
and the claim is proved. Hence we can assume $k=1$.

In this case we have $G=CN$, where $C$ is a Carter subgroup of $G$ and $N$
is a minimal normal subgroup of $G$. Since $G$ is solvable, $N$ is
elementary abelian and in particular, $\left\vert N\right\vert =p^{n}$ for
some prime $p$ and some positive integer $n$. Suppose that $C$ contains a
non-trivial normal subgroup $L$ of $G$. By Lemma \ref{Lem_NContainedInC}, $%
\gamma _{\func{cp}}^{\func{c}}\left( G\right) =\gamma _{\func{cp}}^{\func{c}%
}\left( G/L\right) $ and $G/L=\left( C/L\right) N$ and $N$ is minimal normal
in $G/L$. Thus we can assume that $C$ is core-free. Under this assumption $%
C_{G}\left( N\right) =N$. Indeed, $N\leq C_{G}\left( N\right) $ because $N$
is abelian so by Dedekind's law, $C_{G}\left( N\right) =C_{G}\left( N\right)
\cap \left( NC\right) =N\left( C_{G}\left( N\right) \cap C\right) $. Since $%
C_{G}\left( N\right) \trianglelefteq G$ we get that $C$ normalizes $%
C_{G}\left( N\right) \cap C$. Moreover, $N$ centralizes $C_{G}\left(
N\right) $ hence it normalizes $C_{G}\left( N\right) \cap C$. Thus we proved
that $C_{G}\left( N\right) \cap C\trianglelefteq G$. Since $C$ is core-free,
we get $C_{G}\left( N\right) \cap C=1$ and $C_{G}\left( N\right) =N$.
Finally, since $G$ is solvable, $G$ is primitive iff it has a
self-centralizing minimal normal subgroup (\cite[Proposition A.15.8(b)]%
{DH1992soluble}). Thus we can conclude that $G$ is primitive and $C$ is
maximal and non-normal. Now apply Theorem \ref{Th_AffinePrimitive}, with $%
H=C $.
\end{proof}

\section*{{\protect\huge Appendix} \label{appendix}}

Table \ref{TblSporadics} lists, for each sporadic simple group $S$ including
the Tits group ${}^{2}F_{4}\left( 2\right) ^{\prime }$, an upper bound on $%
\gamma _{\func{cp}}^{p}(S)$ (column heading $\gamma _{\func{cp}}^{p}(S)\leq $%
) for a specified prime $p$. Under column heading $p^{\alpha }$ the maximal
power of $p$ dividing $\left\vert S\right\vert $ is given. Under the column
heading $A$ we specify a subgroup $A<S$ on which the bound is based, using
ATLAS notation (\cite{ATLAS}). We have verified, using Lemma \ref%
{Lem_ss_SporadicTools} (d) and sometimes a MAGMA computation (\cite{MAGMA}),
that $N_{S}\left( P\right) $, the normalizer in $S$ of some Sylow $p$%
-subgroup $P$ of $S$ is solvable. For each $A$ in the table we have $\gamma
_{\func{cp}}^{A}(S)=3$ - this was verified using the "mfer" tool \cite%
{BreuerMuller} (see Subsection \ref{SubSection_SpecialSolvable}). In all
cases, with a few exceptions detailed below (all associated with \ $p=2$), $%
A $ contains a Sylow $p$-subgroup of $S$. For $S=M_{11},J_{1}$, we have $%
P\leq A\leq N_{S}\left( P\right) $ so the bound is exact and $\gamma _{\func{%
cp}}^{p}(S)=3$. For the other cases the bound is derived using Lemma \ref%
{Lem_ss_SporadicTools} (a) and a bound on $\gamma _{\func{cp}}^{p}(A)$ (when 
$\gamma _{\func{cp}}^{p}(S)\leq 9$, we have $\gamma _{\func{cp}}^{p}(A)=3$).
The determination of the bound on $\gamma _{\func{cp}}^{p}(A)$ uses a
variety of means: Lemma \ref{Lem_ss_Lie}, information on subgroups of $A$
from \cite{ATLAS}, an application of the "mfer" tool to $A$, and previous
results from the table. For the $S=B$, where the bound is $12$, we have
deduced $\gamma _{\func{cp}}^{2}(A)\leq 4$ from Theorem \ref%
{Th_unipotent_Sylows}. In this as well as in the case of $S=M$, the argument
relies on Lemma \ref{Lem_ss_SporadicTools} (e),(f), and hence $A$ need not
contain a Sylow $2$-subgroup of $S$. \bigskip

Remarks for Table \ref{TblSporadics}:

$^{\text{(1)}}$ A Sylow $2$-subgroup of $A$ is self-normalizing of index $3$%
, hence $\gamma _{\func{cp}}^{2}(A)=3$.

$^{\text{(2)}}$ A Tomlib mfer calculation shows that $L_{2}\left( 16\right) $
is a product of three Sylow $5$-subgroup normalizers (structure $D_{30}$).
Hence $L_{2}\left( 16\right) :2$ is a product of three Sylow $5$-subgroup
normalizers (structure $D_{10}\times S_{3}$).

$^{\text{(3)}}$,$^{\text{(4)}}$,$^{\text{(6)}}$ $A$ is a group of Lie type
hence it is a product of three Sylow $2$-subgroup normalizers. A MAGMA
computation shows that the Sylow $2$-subgroup of $A$ is self-normalizing.
Therefore $S$ is a product of nine conjugates of a $2$-subgroup and hence it
is a product of nine Sylow $2$-subgroup normalizers.

$^{\text{(5)}}$ $2.HS.2$ is a central extension of $HS.2$ hence the order $2$
center is contained in the Sylow $11$-subgroup normalizer of $2.HS.2$.

\begin{center}
\bigskip 
\begin{table}[H] \centering%
\begin{tabular}{|c|c|c|c|c|}
\hline
$S$ & $A$ & $p^{\alpha }$ & $\gamma _{\func{cp}}^{p}(S)\leq $ & Remarks \\ 
\hline
$M_{11}$ & $11:5$ & $11$ & $3$ &  \\ \hline
$M_{12}$ & $M_{11}$ & $11$ & $9$ &  \\ \hline
$J_{1}$ & $2^{3}:7:3$ & $2^{3}$ & $3$ &  \\ \hline
$M_{22}$ & $L_{2}(11)$ & $11$ & $9$ &  \\ \hline
$J_{2}$ & $U_{3}(3)$ & $3^{3}$ & $9$ &  \\ \hline
$M_{23}$ & $M_{11}$ & $11$ & $9$ &  \\ \hline
${}^{2}F_{4}\left( 2\right) ^{\prime }$ & $2^{2}.\left[ 2^{8}\right] :S_{3}$
& $2^{11}$ & $9$ & $^{\text{(1)}}$ \\ \hline
$HS$ & $M_{11}$ & $11$ & $9$ &  \\ \hline
$J_{3}$ & $L_{2}(16):2$ & $5$ & $9$ & $^{\text{(2)}}$ \\ \hline
$M_{24}$ & $2^{6}:\left( L_{3}\left( 2\right) \times S_{3}\right) $ & $%
2^{10} $ & $9$ &  \\ \hline
$M^{c}L$ & $U_{4}\left( 3\right) $ & $3^{6}$ & $9$ &  \\ \hline
$He$ & $S_{4}(4):2$ & $2^{10}$ & $9$ & $^{\text{(3)}}$ \\ \hline
$Ru$ & ${}^{2}F_{4}\left( 2\right) $ & $2^{14}$ & $9$ & $^{\text{(4)}}$ \\ 
\hline
$Suz$ & $2_{-}^{1+6}.U_{4}\left( 2\right) $ & $2^{13}$ & $9$ &  \\ \hline
$O^{\prime }N$ & $L_{3}\left( 7\right) :2$ & $7^{3}$ & $9$ &  \\ \hline
$Co_{3}$ & $2.S_{6}(2)$ & $2^{10}$ & $9$ &  \\ \hline
$Co_{2}$ & $2_{+}^{1+8}:S_{6}(2)$ & $2^{18}$ & $9$ &  \\ \hline
$Fi_{22}$ & $O_{7}\left( 3\right) $ & $3^{9}$ & $9$ &  \\ \hline
$HN$ & $2.HS.2$ & $11$ & $27$ & $^{\text{(5)}}$ \\ \hline
$Ly$ & $G_{2}\left( 5\right) $ & $5^{6}$ & $9$ &  \\ \hline
$Th$ & $2^{5}.L_{5}\left( 2\right) $ & $2^{15}$ & $9$ &  \\ \hline
$Fi_{23}$ & $S_{8}\left( 2\right) $ & $2^{18}$ & $9$ & $^{\text{(6)}}$ \\ 
\hline
$Co_{1}$ & $2_{+}^{1+8}.O_{8}^{+}\left( 2\right) $ & $2^{21}$ & $9$ &  \\ 
\hline
$J_{4}$ & $2^{11}:M_{24}$ & $2^{21}$ & $27$ &  \\ \hline
$Fi_{24}^{\prime }$ & $3^{7}.O_{7}\left( 3\right) $ & $3^{16}$ & $9$ &  \\ 
\hline
$B$ & $2.^{2}E_{6}(2).2$ & $2^{41}$ & $12$ &  \\ \hline
$M$ & $2.B$ & $2^{46}$ & $36$ &  \\ \hline
\end{tabular}
\caption
{Upper bounds on the minimal length of special solvable conjugate factorzations of simple sporadic groups and the Tits group}%
\label{TblSporadics}%
\end{table}%
\ \ \ 
\end{center}

\end{document}